\documentclass[11pt]{article}
\usepackage{graphicx}
\usepackage{geometry}
\usepackage{adjustbox}

\providecommand{\keywords}[1]{\textbf{\textit{Keywords---}} #1}
\title{
Quantifying Multivariate Graph Dependencies: \\Theory and Estimation for Multiplex Graphs}
\author{Anda Skeja\footnote{Institute of Mathematics, EPFL. Email: {first.last}@epfl.ch} \quad \quad \quad Sofia C. Olhede$^*$}\usepackage[utf8]{inputenc} 
\usepackage{algorithm}
\usepackage{bbm}

\usepackage[T1]{fontenc}
\usepackage{url}              
\usepackage{cite}            
\usepackage[cmex10]{amsmath} 
\usepackage{mleftright}       
\mleftright                   

\usepackage{graphicx}        
\usepackage{booktabs}         
\usepackage{amsmath,amsfonts}
\usepackage{algorithmic}
\usepackage{array}
\usepackage[caption=false,font=normalsize,labelfont=sf,textfont=sf]{subfig}
\usepackage{textcomp}
\usepackage{stfloats}
\usepackage{url}
\usepackage{verbatim}
\usepackage{graphicx}
\usepackage{cite}
\usepackage{amsthm}
\usepackage{booktabs}
\usepackage[font={small,it}]{caption}
\usepackage{accents}

\usepackage[pdfstartview=XYZ,
bookmarks=true,
colorlinks=true,
linkcolor=blue,
urlcolor=blue,
citecolor=blue,
pdftex,
bookmarks=true,
linktocpage=true, 
]{hyperref}

\usepackage{orcidlink}

\usepackage{graphicx}
\DeclareMathOperator*{\argmax}{arg\,max}

\newcommand{\Prb}{{\mathbb{P}}}

\newcommand{\rtwo}{\rho^{(2)}}

\newcommand{\kap}[1]{\rtwo(1)-\lambda^2}

\definecolor{thecolor}{rgb}{0.1,0.5,1}

\newtheorem{theorem}{Theorem}

\newtheorem{cor}{Corollary}

\newtheorem{definition}{Definition}
\newtheorem{example}{Example}
\newtheorem{remark}{Remark}
\definecolor{darkgreen}{rgb}{0.0, 0.2, 0.13}  

\date{}
\begin{document}
\maketitle
\begin{abstract}
Multiplex graphs, characterised by their layered structure, exhibit informative interdependencies within layers that are crucial for understanding complex network dynamics. Quantifying the interaction and shared information among these layers is challenging due to the non-Euclidean structure of graphs. Our paper introduces a comprehensive theory of multivariate information measures for multiplex graphs. We introduce graphon mutual information for pairs of graphs and expand this to graphon interaction information for three or more graphs, including their conditional variants. We then define graphon total correlation and graphon dual total correlation, along with their conditional forms, and introduce graphon $O-$information. We discuss and quantify the concepts of synergy and redundancy in graphs for the first time, introduce consistent nonparametric estimators for these multivariate graphon information--theoretic measures, and provide their convergence rates.  We also conduct a simulation study to illustrate our theoretical findings and demonstrate the relationship between the introduced measures, multiplex graph structure, and higher--order interdependecies. Real-world applications further show the  utility of our estimators in revealing shared information and dependence structures in real-world multiplex graphs. This work not only answers fundamental questions about information sharing across multiple graphs but also sets the stage for advanced pattern analysis in complex networks.

\end{abstract}

\section{Introduction}
\quad 

This paper introduces the quantification of dependency and higher-order interdependencies in multiplex exchangeable random graphs, specifically addressing cases involving two or more graphs. Exchangeability is equivalent to assuming that the distribution of the random graph remains invariant under node permutations~\cite{orbanz2014bayesian}. Due to the non-Euclidean nature of graphs, this challenge requires innovative approaches that combine multivariate information-theoretic measures with the structural properties of graphs. We introduce a comprehensive theory of multivariate information-theoretic measures based on the graph limit formalism, which serves as the generating mechanism for exchangeable graphs. This approach ensures that our measures are not dependent on specific features and are invariant under graph isomorphisms, fulfilling the conditions put forth by~\cite{mowshowitz1968entropy} for complexity measures. We address fundamental questions about the interconnectivity and higher-order interdependence among graphs, exploring how much information is shared among them or a subset of them and whether some graphs can explain relationships among others. We also explore higher-order organization and dependence structures that cannot be traced by lower-order measures, such as synergy and redundancy, and further investigate their relationships with the structure of multiplex graphs.

 Interest in multiplex graphs~\cite{bianconi2018multilayer} has been increasing due to their ability to model multiple types of interactions across a shared set of nodes. These extensions of graphs offer the possibility in unraveling the intricacies of network dynamics, as demonstrated in studies by~\cite{athreya2022discovering,macdonald2022latent,chandna2022edge,zhang2024Consistent}. Building on the framework of exchangeable random graphs, extendable to multivariate cases \cite{chandna2022edge}, our theory also applies to graphs that are subject to graph matching. For further details on graph matching, see \cite{sun2019using, caetano2009learning, caelli2004eigenspace, yan2016short}.

In our previous work~\cite{skeja2023entropy} we defined the complexity of the generating mechanism of any exchangeable random graph via graphon entropy, and introduced a nonparametric graphon entropy estimator. For the multivariate setting, we develop information--theoretic measures in a series of steps: we start by considering two graph observations made on the same node set. We then define the joint entropy and mutual information encompassed by the generating mechanisms of such observations via the bivariate graph limit model~\cite{chandna2022edge}, which we call \emph{joint graphon entropy} and \emph{graphon mutual information}, respectively. Assuming we observe graphs from permutation invariant distributions, we approximate the joint distribution of the two graphs using a two-layer correlated stochastic block model. Through this approach, we estimate the graphon mutual information of bivariate graph observations. The properties of this estimator are discussed, along with its convergence rate. Although mutual information estimation between random variables and vectors is well-explored, with extensive methodologies discussed in~\cite{paninski2003estimation}, estimation of  mutual information between two graphs remains largely unexplored. Initial investigations by~\cite{escolano2017mutual} introduced a method for estimating mutual information between two graphs using copula entropy, which the authors acknowledge is computationally demanding. This underscores a significant gap in the literature regarding a theoretically sound estimator of mutual information between graphs.  Additionally, we introduce the graphon mutual information matrix, where the entries are given by the graphon mutual information between the two graphs under consideration.  

Pairwise dependencies often fail to explain the overall dependence structure when more than two graphs are involved. Therefore, to understand higher-order interactions, we extend our analysis to the trivariate case. The exploration of multivariate dependence measures began with the concept of `interaction information' by~\cite{mcgill1954multivariate}.  The concept of total correlation was first introduced by~\cite{watanabe1960information}, and that of dual total correlation by~\cite{han1975linear}. These measures have been independently identified and renamed on multiple occasions, as evidenced by works such as those by~\cite{joe1989relative,studeny1998multiinformation}. Fundamental research in this direction has focused on discovering methods to quantify various types of dependency structures among multiple random variables. This includes, for example, identifying subsets of variables that demonstrate conditional independence from each other. Early analyses in this direction can be found in works by~\cite{pearl2022graphoids,perez1977varepsilon}. See~\cite{studeny1998multiinformation} for a survey on multivariate dependence measures, and~\cite{paninski2003estimation,rahimzamani2018estimators,mesner2020conditional} for estimation methods. Research into multivariate information-theoretic measures for graphs remains largely unexplored. To address this gap, we introduce \emph{graphon interaction information}, \emph{graphon total correlation}, \emph{graphon dual total correlation}, and \emph{graphon $O-$information}, along with their conditional variants. We also propose consistent estimators for these measures in both bivariate and trivariate settings and detail their convergence rates. Furthermore, we discuss extending these estimators to the $d-$variate setting for $d >3$. Our work establishes a foundational framework for analyzing multiple graph patterns by introducing definitions and estimators for multivariate information-theoretic measures across multiple graphs.

This paper is organized into seven sections. Section~\ref{sec-prelim} provides background on exchangeable random graphs and graph limits. Section~\ref{sec:mutual-information} introduces definitions of the joint graphon entropy and graphon mutual information between two graphs, and the graphon mutual information matrix. Section~\ref{sec-info-measure-three} extends our framework to multivariate settings, introducing measures such as graphon total correlation, graphon dual total correlation, graphon interaction information, and graphon O-information. Section~\ref{sec-est-bivariate} focuses on the estimation of the joint graphon between two and three exchangeable random graphs, along with the introduced information measures. We introduce consistent estimators and provide convergence rates for the bivariate and trivariate settings, as well as discuss extensions to the $d-$variate setting for $d>3$. To demonstrate the practical implications of our findings, Section~\ref{sec-simulation} presents a simulation study that investigates the performance of the introduced estimators and illustrates scenarios of synergy and redundancy through graphon interaction information and graphon total correlation. Additionally, the applicability of our theoretical results is further exemplified through the analysis of real-world graph data in Section~\ref{sec-real-data}. Finally, we offer some conclusions in Section~\ref{sec-conclusion}.

\section{Notation and preliminaries}\label{sec-prelim}
In this work, all graphs are undirected and simple, i.e without self-loops or multiple edges. The vertex set of a graph $G$ is represented as $V(G)$, and its edge set as $E(G)$. Nodes $i$ and $j$ are considered adjacent if there is an edge connecting them. For a graph $G$ with $n$ vertices, its adjacency matrix is denoted by $A\in \{0,1\}^{n \times n}$, where the entry $A_{ij}$ is set to one if there is an edge between nodes $i$ and $j$, and zero if not. The degree of a vertex $i$ is indicated by $d_i$. We use $\log(\cdot)$ to denote the natural logarithm throughout the paper. We use the notation $W^{(d)}$ to denote the graphon giving rise to the exchangeable random graph $G_d(n,W^{(d)})$. We note that the multivariate graphon information--theoretic measures that will shortly be introduced are functions of the underlying graph limit and independent of the graphon representing them, similar to graphon entropy as discussed in~\cite{hatami2018graph}. Therefore, we will use the terms graph limit and graphon interchangeably in this context. Additionally, $\mathbf{W}_{1,2,\ldots,d}$ represents the system of graphons with $2^d-1$ entries, where each entry corresponds to a parameter of the $d$-variate graph limit model. We let $\mathrm{MultBern_d}$ denote the $d-$variate Bernoulli distribution, for details see~\cite{teugels1990some}.
\subsection{Graphons and exchangeable random graphs}
\begin{definition}[Graphon~\cite{lovasz2006limits}]\label{def:graphon}
    A graphon $W:[0,1]^2 \to [0,1]$ is a symmetric measurable function.
\end{definition}
A random graph $G$ is considered exchangeable if its adjacency matrix constitutes a jointly exchangeable array, and it can be defined in terms of the random graphon defined in Definition~\ref{def:graphon} . This category of graphs has received considerable attention, as exemplified by works like~\cite{lauritzen2008exchangeable,orbanz2014bayesian,austin2015exchangeable,janson2013graphons}, and can be notably characterized by a foundational representation theorem known as the Aldous-Hoover theorem~\cite{aldous1981representations,hoover1979relations}.
\begin{definition}[Joint exchangeability]
A random array $(A_{ij})_{i,j \in \mathbb{N}^*}$ is jointly exchangeable if for any permutation $\pi$ of $\mathbb{N}^*$ 
\begin{equation*}
(A_{ij})\overset{d}{=}(A_{\pi(i)\pi(j)}).
\end{equation*}
\end{definition}
\begin{theorem}[Aldous--Hoover~\cite{aldous1981representations,hoover1979relations}]\label{thm-aldous}
Let $A$ be a jointly exchangeable random array. Then there exists an i.i.d. sequence $\xi=(\xi_1,...,\xi_n)$ following $\mathrm{U}(0,1)$, a random variable $\gamma \sim \mathrm{U}(0,1)$ independent of $\xi$, and a random function $W:[0,1]^3 \to [0,1]$ such that \begin{equation*}\label{graphon}
    \Prb(A_{ij}=1 | \xi, \gamma) = W(\xi_i,\xi_j,\gamma),
\end{equation*} and $A_{ij}$ are conditionally independent across $i,j$ given $\xi$ and $\gamma$.
\end{theorem}\label{thm-aldous-hoover}
To represent any infinite array, we need the additional random variable $\gamma$, but
as we shall assume, we only have one realization of $A$. Thus, the dependence on $\gamma$ cannot be estimated, and we shall suppress that dependence, assuming the graph is disassociated.
Therefore, we will use $W(\xi_i,\xi_j)$ for $1\leq i,j\leq n=|V|$ throughout this paper. 
\begin{example}[Exchangeable graph]\label{example-exch-graph}
Let $G=(V,E)$ with $|V|=n$ be a random graph and $A\in \{0,1\}^{n \times n}$ be its adjacency matrix. Then, the variable $\xi_i$ is associated with vertex $i$, and $W(\xi_i,\xi_j)$ in Theorem \ref{thm-aldous} is the graphon giving rise to the exchangeable graph $G(n,W)$.
\end{example}
\begin{definition}[Graphon entropy~\cite{janson2013graphons,hatami2018graph}]\label{def-entropy-hatami}
 The entropy of a graphon $W(x,y)$ is defined to be
\begin{equation}
\label{eqn:Hatemi-ent}
{\mathcal H}(W)=\iint_{[0,1]^2} h(W(x,y))\,dx\,dy,     
\end{equation}
where $h:[0,1]\rightarrow \mathbb{R}_{+}$ denotes the binary entropy given as
\begin{equation}
    \label{eqn:binary}
    h(x)=-x\log(x)-(1-x)\log(1-x),
\end{equation}\label{binaryentropy}
for $x\in [0,1]$, where $h(0)$ and $h(1)$ are defined to be zero to ensure continuity.
\end{definition}
\begin{definition}\label{assump:smooth}($\alpha$-H\"older graphon) We say that the exchangeable array $A\in \{0,1\}^{n\times n}
    $ is generated by an $\alpha-$H\"older graphon if $W \in \mathcal{F}^{\alpha}(M)$, where the latter is the class of H\"older functions with exponent $\alpha$ detailed as follows:\begin{align*}\label{holder-graphon}
    W\in\mathcal{F}^{\alpha}(M)\iff \sup_{(x,y)\neq(x',y')\in(0,1)^2} \frac{|W(x,y)-W(x',y')|}{|(x,y)-(x',y')|^{\alpha}}\leq M<\infty.
\end{align*}
\end{definition}

\section{Graphon mutual information between two graphs}\label{sec:mutual-information}

Mutual information, detailed in~\cite{cover1999elements}, captures all forms of dependence between two random variables, indicating that two variables are independent if and only if their mutual information is zero. This measure adheres to fundamental information-theoretic principles and can be computed as the difference between the sum of the individual entropies and the joint entropy. Mutual information has also been used to evaluate community detection algorithms (e.g.,~\cite{mayya2019mutual,newman2020improved,luo2020highly,zhang2015evaluating,jerdee2023normalized,jerdee2024mutual}). Furthermore, given that the stochastic block model corresponds to block-constant graphons, this approach can be applied to the setting of stochastic block models as well.

Calculating the mutual information between two graphs requires establishing their joint distribution and subsequently computing their respective entropies. Technically, defining the entropy of a graph necessitates establishing a probability distribution for a specific graph feature, an approach that is often limited due to the resulting entropy being feature-dependent~\cite{zenil2014correlation}. The authors of~\cite{skeja2023entropy, skeja2023measuring} have utilized graphon entropy, a graph property that is invariant under graph isomorphisms~\cite{hatami2018graph}, as a favorable measure to characterize the complexity of exchangeable graphs. We employ the bivariate graph limit model~\cite{chandna2022edge} to extend the univariate case from~\cite{skeja2023entropy}, proposing a method wherein the joint entropy of two graphs' generating mechanisms is conceptualized as the entropy associated with the bivariate graph limit model. We begin by discussing the bivariate graph limit model, then proceed to introduce joint graphon entropy, and consequently, graphon mutual information.

\subsection{Bivariate graph limit model}
Let $G_1(n,W^{(1)})$ and $G_2(n,W^{(2)})$ denote two exchangeable graphs on $n$ nodes generated by graphons $W^{(1)}$ and $W^{(2)}$ with a common latent vector $\xi \in (0,1)^n$ such that $\xi_i \sim U(0,1)$ for all $i \in [n]$, and let $A^{(1)}$ and $A^{(2)}$ denote the adjacency matrices of $G_1(n,W^{(1)})$ and $G_2(n,W^{(2)})$, respectively. Define $\underline{A}_{ij}=[A_{ij}^{(1)} \ A_{ij}^{(2)}]^T$ as the vector observation. It follows from~\cite{teugels1990some} that two marginal moments, namely, $\mathbb{E}[A_{ij}^{(1)}]$ and $\mathbb{E}[A_{ij}^{(2)}]$, and a co-dependence measure between them fully specify the distribution. A natural uncentered co-dependence measure, defined by~\cite{chandna2022edge}, can be obtained by the Hadamard product between $A^{(1)}$ and $A^{(2)}$, as follows:
\begin{equation*}\label{A-12}
    A^{(12)}_{ij}=(A^{(1)}\circ A^{(2)})_{ij}=A_{ij}^{(1)}A_{ij}^{(2)}, \quad 1\leq i,j \leq n.
\end{equation*}
Consequently, the expectation of this product corresponds to the joint probability,
\begin{equation}\label{f12}
    \mathbb{E}[A_{ij}^{(1)} A_{ij}^{(2)}]=\mathbb{P}(A_{ij}^{(1)}=1 , A_{ij}^{(2)}=1).
\end{equation}
Finally, the distribution of
$\underline{A}_{ij}$ takes the following form~\cite{chandna2022edge}:
    \begin{equation*}
    \underline{A}_{ij}|\xi \sim \mathrm{MultBernoulli_2}(\underline{W}(\xi_i, \xi_j)),
\end{equation*}
where $\underline{W}(\cdot)=[W^{(1)}(\cdot) \ \  W^{(2)}(\cdot) \ \ W^{(12)}(\cdot)]^T$. The graphon $W^{(12)}$ is as defined in Equation~\eqref{f12}.

\begin{table*}[t]
  \centering
  \begin{tabular}{cccc}
 & $A^{(2)}_{ij}=1$ & $A^{(2)}_{ij}=0$ \\ 
 $A^{(1)}_{ij}=1$ & $W^{(12)}(\xi_i,\xi_j)$ & $W^{(1)}(\xi_i,\xi_j)-W^{(1,2)}(\xi_i,\xi_j)$  \\  
 $A^{(1)}_{ij}=0$ & $W^{(2)}(\xi_i,\xi_j)-W^{(12)}(\xi_i,\xi_j)$ & $1-W^{(1)}(\xi_i,\xi_j)-W^{(2)}(\xi_i,\xi_j)+W^{(12)}(\xi_i,\xi_j)$ 
  \end{tabular}
  \caption{Bivariate Bernoulli distribution of two exchangeable adjacency matrices}
  \label{table}
\end{table*}

Consider the system of graphons given as 
\begin{align}\label{eq-system-graphons-bivariate}
    \mathbf{W}_{1,2}=[W^{(12)} \quad (W^{(1)}-W^{(12)}) \quad (W^{(2)}-W^{(12)}) (1-W^{(1)}-W^{(2)}+W^{(12)})],
\end{align} where each entry is a graphon resembling a parameter of the bivariate Bernoulli distribution as detailed in Table~\ref{table}. One possible way to generate graphons $W^{(1)}, W^{(2)}$ that are related and then find $W^{(12)}$ is via the input--output method introduced in~\cite{chandna2022edge}, which we will further detail in Section~\ref{sec-simulation}.
 We note that each entry is non-negative and the entries of this vector sum to one, therefore we can define the Shannon entropy of this system to be as follows:
\begin{definition}[Bivariate joint graphon entropy]\label{def-joint-graphon-entropy} Let $\mathbf{W}_{1,2}$ be as given in~\eqref{eq-system-graphons-bivariate}. We define the joint graphon entropy as follows:
\begin{align*}
   \nonumber  \mathcal{H}(\mathbf{W}_{1,2})=
   -\iint_{[0,1]^2} &\Big\{ W^{(12)} \log W^{(12)}
   +(W^{(1)}-W^{(12)})\log (W^{(1)}-W^{(12)})\\
   &+(W^{(2)}-W^{(12)})\log (W^{(1)}-W^{(12)})\\
   &+(1-W^{(1)}-W^{(2)}+W^{(12)})\log (1-W^{(1)}-W^{(2)}+W^{(12)}) \Big\} \ dx \ dy.
 \end{align*}
   \end{definition} 
\begin{cor}\label{thm-graphs-graphon} Let $G_1(n,W^{(1)})$ and $G_2(n,W^{(2)})$ be two exchangeable random graphs on $n$ nodes generated by the graphons $W^{(1)}$ and $W^{(2)}$, respectively on the probability space $(\Omega,\mu)$, where $\Omega=(0,1)^2$ and $\mu=U(0,1)$. Then, as $n\to \infty$
    \begin{equation*}
        \frac{\mathcal{H}(G_1(n,W^{(1)}), G_2(n,W^{(2)}))}{{n \choose 2}} \to \mathcal{H}(\mathbf{W}_{1,2})
    \end{equation*}
\end{cor}
\begin{proof}See Appendix I. A.
\end{proof}
Corollary~\ref{thm-graphs-graphon}, which follows from the univariate case shown in~\cite{janson2013graphons}, establishes that as $n$ grows, the joint entropy of two graphs normalized by ${n \choose 2}$ converges to the entropy of the system of graphons as defined in Definition~\ref{def-joint-graphon-entropy}. Consequently, we use the term \emph{joint graphon entropy} as introduced in Definition~\ref{def-joint-graphon-entropy}. Similarly, this convergence applies, wlog, to the joint entropy of any $d$ graphs sharing the same $n$ and the latent vector $\xi$. 
\begin{definition}[Graphon mutual information] Considering the definition of the bivariate joint graphon entropy given in Definition~\ref{def-joint-graphon-entropy}, the graphon mutual information between two exchangeable graphs $G_1(n,W^{(1)})$ and $G_2(n,W^{(2)})$ can then be defined as follows:
\begin{equation*}\label{joint_entropy}
\mathrm{I}_{\mathbf{W}_{1,2}}(W^{(1)};W^{(2)})=\mathcal{H}(W^{(1)})+\mathcal{H}(W^{(2)})-\mathcal{H}(\mathbf{W}_{1,2}),
\end{equation*}
where $\mathcal{H}(W^{(1)})$ and $\mathcal{H}(W^{(2)})$ as defined in Definition~\ref{def-entropy-hatami} denote the entropies of graphons $W^{(1)}$ and $W^{(2)}$, respectively.
\end{definition}
Alternatively the `distance' between two graph generating mechanisms could be denoted as
\begin{equation}\label{distance}
    \mathrm{D}_{\mathbf{W}_{1,2}}(W^{(1)};W^{(2)})=\mathrm{I}_{\mathbf{W}_{1,2}}(W^{(1)};W^{(1)})-\mathcal{H}(\mathbf{W}_{1,2}).
\end{equation}
Note that ~\eqref{distance} satisfies the standard axioms for a distance.
Recall that mutual information quantifies the KL divergence between the joint distribution of two variables and the product of their marginal distributions. Therefore, it represents a special case of total correlation, which we will introduce for graphs in subsequent sections.

\subsection{Graphon mutual information matrix}
Summarizing the pairwise dependencies between graph-generating mechanisms of layers of multiplex graphs can be naturally achieved through a \emph{graphon mutual information matrix}. This matrix also helps in understanding the interconnectedness of the multiplex graph through its associated entropy, which will be explained further below.  To quantify the structure within these pairwise connections, we utilize the von Neumann entropy, particularly useful in scenarios where understanding the organization in pairwise relationships is important. The von Neumann entropy of a matrix summarizing pairwise dependencies was used by~\cite{felippe2023threshold} to address the open problem of calculating entropy from the Pearson correlation matrix without any thresholding steps. The mutual information matrix for time series has previously been defined in~\cite{liu2007brain}, used in~\cite{zhao2017mutual}, and studied by~\cite{jakobsen2014mutual}.
\begin{definition}[Graphon mutual information matrix]\label{def-mim}
    Let $G_1(n,W^{(1)}),\ldots,G_d(n,W^{(d)})$ be $d$ graphs on the same set of vertices. Then, we define the $d \times d$ graphon mutual information matrix to account for pairwise dependencies between graph generating mechanisms as follows:
\begin{align*}
    \mathbf{I}=\begin{pmatrix}
          \mathcal{H}(W^{(1)}) & \mathrm{I}_{\mathbf{W}_{1,2}}(W^{(1)};W^{(2)}) & \cdots & \mathrm{I}_{\mathbf{W}_{1,d}}(W^{(1)};W^{(d)}) \\
         \mathrm{I}_{\mathbf{W}_{2,1}}(W^{(2)};W^{(1)}) & \mathcal{H}(W^{(2)}) & \cdots & \mathrm{I}_{\mathbf{W}_{2,d}}(W^{(2)};W^{(d)}) \\
          \vdots & \vdots & \ddots & \vdots\\
         \mathrm{I}_{\mathbf{W}_{d,1}}(W^{(d)};W^{(1)}) & \mathrm{I}_{\mathbf{W}_{d,2}}(W^{(d)};W^{(2)}) & \cdots & \mathcal{H}(W^{(d)}) \\
    \end{pmatrix}
\end{align*}
\end{definition}
This matrix is symmetric since $\mathrm{I}_{\mathbf{W}_{i,j}}(W_i,W_j)=\mathrm{I}_{\mathbf{W}_{j,i}}(W_j,W_i)$ for $1\leq i,j\leq n$. Its entries are also all nonnegative as mutual information takes values in $\mathbb{R}^{+}$, with value $0$ indicating independence. We note that $\mathrm{I}_{\mathbf{W}_{i,i}}(W^{(i)};W^{(i)})=\mathcal{H}(W^{(i)})$. For a more intuitive interpretation, we normalize the entries of the graphon mutual information matrix and let the normalized entries form the \emph{normalized graphon mutual information matrix} $\mathbf{I}^{\mathrm{N}}$ as follows:
\begin{align}\label{normalized-mi}
    \mathrm{I}_{\mathbf{W}_{i,j}}^{\mathrm{N}}(W_i;W_j)=\frac{\mathrm{I}_{\mathbf{W}_{i,j}}(W^{(i)};W^{(j)})}{\min\{\mathcal{H}(W^{(i)}),\mathcal{H}(W^{(j)})\}}, \quad 1\leq i,j \leq n.
\end{align}
This normalization ensures that $\mathrm{I}_{\mathbf{W}_{i,i}}(W^{(i)};W^{(i)})=1$ attains the maximum value and all normalized mutual information values lie in the range $[0,1]$.

Following~\cite{felippe2023threshold}, a natural measure of the order present in the pairwise dependencies of the system of graphs is entropy. The order present in the pairwise interdependencies of the system of graphs can be summarized by the \emph{von Neumann entropy} of the graphon mutual information matrix, which needs to be normalized appropriately. This normalization will ensure that the graphon mutual information matrix resembles a density matrix, whose von Neumann entropy can consequently be calculated via its eigenvalues. The density matrix must be symmetric, have a unit trace, and be positive semi-definite. The first condition of symmetry is automatically satisfied by the definition of mutual information. To fulfill the second condition of having a unit trace, we normalize the entries of the mutual information matrix by the size of the layers and define $\mathbf{I}^{\rho}$ to be the \emph{graphon mutual information density matrix} as follows:
\begin{equation}\label{density-matrix}
   \mathbf{I}^{\rho}=\frac{1}{d}\mathbf{I}^{\mathrm{N}},
\end{equation}
where each entry is a normalized graphon mutual information divided by the dimension of the matrix, denoted as $d$ for a $d \times d$ matrix. Here, the superscript $\rho$ merely indicates that the normalized graphon mutual information matrix resembles a density matrix. In this manner, the graphon mutual information density matrix will be symmetric and also possess a unit trace. As previously stated, the last condition for a density matrix is positive semi-definiteness. Contrary to intuition, it has been proved that not all mutual information matrices are positive semi-definite~\cite{jakobsen2014mutual}. Positive semi-definiteness for the mutual information matrix as given in Definition~\ref{def-mim} is proven to hold when the matrix's dimension is $3\times 3$. Consequently, the same condition holds true for the graphon mutual information density matrix. Therefore, the entropy of the density matrix, which will be discussed below, can be applied to the case of three graphs and should be applied to more than three graphs by first verifying that the eigenvalues are non-negative. Once positive semi-definiteness of the graphon mutual information density matrix is confirmed, its entropy can be calculated as the von Neumann entropy of a density matrix.
\begin{definition}[Entropy of the graphon mutual information density matrix] Let the graphon mutual information density matrix be given as in~\eqref{density-matrix}. Then, the von Neumann entropy of $\mathbf{I}^{\rho}$ takes the following form
    \begin{equation*}
    \mathcal{H}(\mathbf{I}^{\rho})=-\sum_{i=1}^d \lambda_i \log(\lambda_i),
\end{equation*}
where $\lambda_i$ for $i\in[d]$ denotes the i$^{th}$ eigenvalue and $\log(\cdot)$ denotes the natural logarithm.
\end{definition}
This entropy definition achieves a maximum value of $\log d$. Thus, dividing by $\log d$ ensures that the entropy lies within the range of $[0,1]$. Therefore, we will use the following normalized entropy for the graphon mutual information density matrix:
\begin{equation}\label{von-Neumann-normalization}
\mathcal{H}^{\mathrm{N}}(\mathbf{I}^{\rho}) = \frac{\mathcal{H}(\mathbf{I}^{\rho})}{\log(d)}.
\end{equation}
The entropy of the graphon mutual information density matrix provides insights into the extent of the pairwise interconnectedness of the layers of the multiplex graph. Consider the case where the graphon mutual information between any two graphs $G_i(n,W^{(i)})$ and $G_j(n,W^{(j)})$ for $i,j \in [d]$ is zero. In this case, the graphon mutual information density matrix would be $\mathbf{I}^{\rho}=\frac{\mathbf{I}}{d}$, where $\mathbf{I}$ denotes the identity matrix. The entropy of this density matrix would equal $\log d$, which is the maximal entropy for a 
$d\times d$ density matrix, corresponding to a \emph{completely mixed state}. Conversely, if all graphs are perfectly dependent on each other, the resulting entropy would be zero, corresponding to a \emph{pure state configuration}. In conclusion, the von Neumann entropy of the graphon mutual information density matrix quantifies the pairwise interconnectedness of the multiplex graph.

\section{Multivariate graphon information--theoretic measures}\label{sec-info-measure-three}

Multivariate information-theoretic measures for graphs remain largely unexplored. Even in the realm of scalar random variables, extending mutual information to higher dimensions has proven to be a non-trivial task, resulting in various definitions of multivariate information measures that exhibit subtle yet significant differences~\cite{timme2014synergy}. While information-theoretic approaches for one or two variables are well-established and extensively studied, many real-world systems involve higher-order interdependencies among three or more random variables~\cite{rosas2019quantifying}. To tackle this complexity, efforts have been made to develop multivariate information-theoretic measures aimed at capturing phenomena such as `synergy' and `redundancy'. However, arriving at a consensus on the precise definitions of synergy remains an ongoing challenge~\cite{williams2010nonnegative, brenner2000synergy, Griffith2014}. We will briefly review the concepts relevant to this paper below and do not intend to provide a comprehensive overview of multivariate information-theoretic measures.

\subsubsection*{Short background on multivariate mutual information, total correlation, dual total correlation, and the O-information}
The pioneering work of~\cite{mcgill1954multivariate} introduced the concept of interaction information, representing the earliest endeavor to quantify the relationship among three variables within a joint probability distribution. Contrary to mutual information, it can be both positive and negative. McGill's original interpretation~\cite{mcgill1954multivariate} states that the interaction information is ``the gain (or loss) in sample information transmitted between any two of the variables, due to the additional knowledge of the third variable''. This measure has been extensively used in literature and has been interpreted in terms of synergy and redundancy~\cite{brenner2000synergy,anastassiou2007computational,williams2010nonnegative}. As an extension of the interaction information, authors of~\cite{chechik2001group} have introduced the Redundancy--Synergy Index (RSI). Authors of~\cite{rosas2019quantifying} differentiate higher order dependencies as ``collective constraints'' and ``shared randomness''. They define the $O$-information, previously introduced as the ``enigmatic information'' by~\cite{james2011anatomy} to capture both synergy and redundancy. It differentiates between scenarios dominated by redundancy, where three or more variables contain the same information, and systems dominated by synergy, which are defined by complex patterns at higher orders that cannot be deduced from simpler, lower-order interactions~\cite{rosas2019quantifying}. This measure is defined as the difference between the \emph{total correlation}, representing collective constraints, and the \emph{dual total correlation}, representing shared randomness. It is equivalent to the interaction information for the case of three random variables~\cite{rosas2019quantifying}. Total correlation~\cite{watanabe1960information} for $d$ random variables is the KL divergence between their joint distribution and the product of their marginals. It assesses the discrepancy between the joint distribution and the product of their marginal distributions, and has emerged as an attempt of having a multivariate information measure that is positive. It is crucial in the context of `optimal dependence structure simplification'~\cite{watanabe1960information}.  
Dual total correlation introduced by~\cite{han1978nonnegative}, measures the information shared among multiple variables, specifically the information that becomes available when observing more than one variable. This measure is alternatively known as excess entropy~\cite{olbrich2008should}, or binding information~\cite{vijayaraghavan2017anatomy,abdallah2012measure}.

 \subsection{Multivariate graph limit model}
We model the distribution of $d$ exchangeable random graphs through their generating mechanisms using a 
$d-$variate Bernoulli distribution, denoted by $\mathrm{MultBernoulli_d}$ which yields a $d-$variate graph limit model~\cite{chandna2022edge}. Characterizing the full distribution requires 
$2^{d}-1$ moments. On the other hand, considering only pairwise moments and marginal distributions requires $d(d+1)/2$ moments, sufficient only when $d=2$.  For scenarios where $d>2$, accurately capturing higher-order dependencies requires  multivariate information-theoretic measures that consider the entire $d$-variate distribution. This approach ensures that all higher--order interdependencies within the system are represented. Building on the definition of graphon mutual information, we define graphon interaction information, graphon total correlation, graphon dual total correlation, and graphon $O-$information among multiple graphs via the multivariate graph limit model.  We consider the vector adjacency matrix observation $\underline{A}_{ij}$ for each pair of nodes $i$ and $j$, conditioned on the latent vector $\xi$ that follows a $d-$variate Bernoulli distribution with parameter $\underline{W}(\xi_i, \xi_j)$. Specifically, the multivariate graph limit model is defined as follows:
\begin{equation*}
\underline{A}_{ij}|\xi \sim \mathrm{MultBernoulli_d}(\underline{W}(\xi_i, \xi_j)),
\end{equation*}
where $\underline{A}_{ij}=[A^{(1)}_{ij} \ldots A^{(d)}_{ij}]$ and $\underline{W}$ is defined as:
\begin{equation*}
\underline{W} = [W^{(1)} \ W^{(2)} \ \ldots \ W^{(d)} \ W^{(12)} \ W^{(13)} \ \ldots \ W^{(d-1 \ d)} \ \ldots]^T,
\end{equation*}
and contains $2^d - 1$ parameters. These parameters capture all possible combinations of interactions among the $d$ graphs. For more details on the formulation and implications of this model, see~\cite{chandna2022edge}. 
\newline
When $d=3$, the multivariate graph limit model is given by:
    \begin{equation*}
    \underline{A}_{ij}|\xi \sim \mathrm{MultBernoulli_3}(\underline{W}(\xi_i, \xi_j)),
\end{equation*}
where $\underline{W}(\cdot)=[W^{(1)}(\cdot) \quad W^{(2)}(\cdot) \quad W^{(3)} \quad W^{(12)}(\cdot) \quad W^{(13)}(\cdot) \quad W^{(23)}(\cdot) \quad W^{(123)}(\cdot)]^T$. We specify the full distribution of the trivariate graph limit model by the equations outlined below, where now we will need $2^3-1$ moments:
\begin{align}    \label{eq-parameter-trivariate}\Prb(A_{ij}^{(1)}=1,A_{ij}^{(2)}=1,A_{ij}^{(3)}=1|\xi)&=W^{(123)}(\xi_i,\xi_j)\\
\Prb(A_{ij}^{(1)}=1,A_{ij}^{(2)}=1,A_{ij}^{(3)}=0|\xi)&=W^{(12)}(\xi_i,\xi_j)-W^{(123)}(\xi_i,\xi_j)\\    \Prb(A_{ij}^{(1)}=1,A_{ij}^{(2)}=0,A_{ij}^{(3)}=1|\xi)&=W^{(13)}(\xi_i,\xi_j)-W^{(123)}(\xi_i,\xi_j)\\
\Prb(A_{ij}^{(1)}=0,A_{ij}^{(2)}=1,A_{ij}^{(3)}=1|\xi)&=W^{(23)}(\xi_i,\xi_j)-W^{(123)}(\xi_i,\xi_j)\\
\Prb(A_{ij}^{(1)}=1,A_{ij}^{(2)}=0,A_{ij}^{(3)}=0|\xi)&=W^{(1)}(\xi_i,\xi_j)-W^{(12)}(\xi_i,\xi_j)-W^{(13)}(\xi_i,\xi_j)+W^{(123)}(\xi_i,\xi_j)\\
\Prb(A_{ij}^{(1)}=0,A_{ij}^{(2)}=1,A_{ij}^{(3)}=0|\xi)&=W^{(2)}(\xi_i,\xi_j)-W^{(12)}(\xi_i,\xi_j)-W^{(23)}(\xi_i,\xi_j)+W^{(123)}(\xi_i,\xi_j)\\
\Prb(A_{ij}^{(1)}=0,A_{ij}^{(2)}=0,A_{ij}^{(3)}=1|\xi)&=W^{(3)}(\xi_i,\xi_j)-W^{(13)}(\xi_i,\xi_j)-W^{(23)}(\xi_i,\xi_j)+W^{(123)}(\xi_i,\xi_j)
\end{align}
Then, the last case, i.e. \begin{align}\label{eq-parameter-trivar}
\nonumber &\Prb(A_{ij}^{(1)}=0,A_{ij}^{(2)}=0,A_{ij}^{(3)}=0)\\
&=1-W^{(1)}(\xi_i,\xi_j)-W^{(2)}(\xi_i,\xi_j)-W^{(3)}(\xi_i,\xi_j)+W^{(12)}(\xi_i,\xi_j)+W^{(13)}(\xi_i,\xi_j)+W^{(23)}(\xi_i,\xi_j)-W^{(123)}(\xi_i,\xi_j)
\end{align} 
follows from the law of total probability. 

Now, we can define the \emph{joint graphon entropy of the trivariate graph limit model} by considering a system of graphons again, and denoting it by $\mathbf{W}_{1,2,3}$, where $\mathbf{W}_{1,2,3}$ has $8$ entries as given by 
\begin{align}\label{eq-system-graphon-trivariate}
    \nonumber  &{\mathbf{W}}_{1,2,3}= 
    \Big[\overbrace{W^{(123)}}^{:=p_{111}} \quad \overbrace{({W}^{(12)}-{W}^{(123)})}^{:=p_{110}} \quad \overbrace{({W}^{(13)}-{W}^{(123)})}^{:=p_{101}} \quad \overbrace{({W}^{(23)}-{W}^{(123)})}^{:=p_{011}} \\
  \nonumber  &\overbrace{({W}^{(1)}-{W}^{(12)}-{W}^{(13)}+{W}^{(123)})}^{:=p_{100}}\quad  \overbrace{({W}^{(2)}-{W}^{(12)}-{W}^{(23)}+{W}^{(123)})}^{:=p_{010}} \\
   &\overbrace{({W}^{(3)}-{W}^{(13)}-{W}^{(23)}+{W}^{(123)})}^{:=p_{001}} \quad \overbrace{(1-{W}^{(1)}-{W}^{(2)}+{W}^{(3)}+{W}^{(12)}+{W}^{(13)}+{W}^{(23)}-{W}^{(123)})}^{:=p_{000}}\Big].
\end{align}
Each entry of $\mathbf{W}_{1,2,3}$ as given in~\eqref{eq-system-graphon-trivariate}, is a graphon resembling a parameter of the trivariate Bernoulli distribution, detailed in Equations~\eqref{eq-parameter-trivariate}--\eqref{eq-parameter-trivar}, and they sum to one. This property enables the definition of an entropy for the system. We use the notation $p_{ijk}$ for $i, j, k \in \{0, 1\}$, as used by~\cite{janson2008graph} previously, for simplification purposes.
\begin{definition}[Trivariate joint graphon entropy]Let $\mathbf{W}_{1,2,3}$ be as given in~\eqref{eq-system-graphon-trivariate}. We define the  joint graphon entropy among three exchangeable graphs $G_1(n,W^{(1)})$, $G_2(n,W^{(2)}), G_3(n,W^{(3)})$ as follows: 
    \begin{align*}
    \mathcal{H}(\mathbf{W}_{1,2,3})=-\iint_{[0,1]^2}&\Big\{p_{111}\log p_{111}+p_{110}\log p_{110}+p_{101}\log p_{101}\\
    &+p_{011}\log p_{011}+p_{100}\log p_{100}+p_{010}\log p_{010}+p_{001}\log p_{001}+p_{000}\log p_{000}\Big\}\ dx \ dy. 
    \end{align*}
\end{definition}
The joint graphon entropy for the $d-$variate setting can be analogously defined by specifying the full distribution of the $d-$variate graph limit model.
\subsection{Graphon interaction information}
Similarly to the bivariate case, the interaction information can be expressed as a linear combination of the individual entropies, the joint bivariate entropies, and the joint trivariate entropy. Unlike  mutual information, which is often straightforwardly interpreted as the Kullback-Leibler (KL) divergence between the joint distribution of two variables and the product of their marginal distributions, the trivariate mutual information among three random variables $X, Y, Z$ involves a more complex relationship, and is not defined as the KL divergence between the joint distribution of the three random variables and the products of the marginal distributions, respectively~\cite{cover1999elements}. It is defined as:  $$\mathrm{I}(X;Y;Z)=\mathcal{H}(X)+\mathcal{H}(Y)+\mathcal{H}(Z)-(\mathcal{H}(X;Y)+\mathcal{H}(X;Z)+\mathcal{H}(Y;Z))+\mathcal{H}(X;Y;Z).$$ 

For $d=3$, the interaction information equals the O-information as defined by \cite{rosas2019quantifying}. In their interpretation, negative values of the O-information, and thus the interaction information for $d=3$, indicate synergy, while positive values indicate redundancy. In the trivariate context, the interaction information quantifies not only shared information, but also the higher-order organizational structure among the three variables, which includes potential redundancies and synergies not apparent in lower-dimensional analyses. Authors of~\cite{timme2014synergy} illustrate the difference between synergy and redundancy with the following example: consider a basic system where any two random variables, $X_1$ and $X_2$, offer insights about a third random variable, $Y$. In essence, if the states of $X_1$ and $X_2$ are known, this gives us some knowledge about $Y$. Specifically, the information that cannot be obtained by knowing $X_1$ and $X_2$ individually but can be understood when both are considered together is described as being provided \emph{synergistically} by $X_1$ and $X_2$. This synergy represents the additional information gained by considering $X_1$ and $X_2$ jointly rather than separately. A similar concept applies to redundancy. Again, assume $X_1$ and $X_2$ give some information about $Y$. The overlapping information that $X_1$ and $X_2$ individually provide is considered to be redundant. Redundancy is the information that is shared by both $X_1$ and $X_2$.
\begin{remark}For Definitions~\ref{def-trivariate-mi}--\ref{def-cond-tc}, let $\mathbf{W}_{1,2,3}$ as given in~\eqref{eq-system-graphon-trivariate} denote the system of graphons resembling the parameters of the trivariate Bernoulli distribution and let $\mathbf{W}_{i,j}$ for $1\leq i,j \leq 3$ as given in~\eqref{eq-system-graphons-bivariate} denote the system of graphons resembling parameters of the bivariate Bernoulli distribution. 
\end{remark}

\begin{definition}[Graphon interaction information]\label{def-trivariate-mi}We define the graphon interaction information among three exchangeable graphs $G_1(n,W^{(1)})$, $G_2(n,W^{(2)})$, $G_3(n,W^{(3)})$ as follows:
    \begin{align*}
        \mathrm{I}_{\mathbf{W}_{1,2,3}}(W^{(1)};W^{(2)};W^{(3)})&=\mathcal{H}(W_1)+\mathcal{H}(W_2)+\mathcal{H}(W_3)\\
        &-(\mathcal{H}(\mathbf{W}_{1,2})+\mathcal{H}(\mathbf{W}_{1,3})+\mathcal{H}(\mathbf{W}_{2,3}))\\
        &+\mathcal{H}(\mathbf{W}_{1,2,3}).
    \end{align*}
\end{definition}
\begin{definition}[Conditional graphon mutual information] We define the conditional graphon mutual information between two exchangeable graphs $G_1(n,W^{(1)})$, $G_2(n,W^{(2)})$ given the generating mechanism of the third graph $G_3(n,W^{(3)})$ as follows:
\begin{equation*}
    \mathrm{I}_{\mathbf{W}_{1,2,3}}(W_1;W_2|W_3)=\mathcal{H}(\mathbf{W}_{1,3})+\mathcal{H}(\mathbf{W}_{2,3})-\mathcal{H}(W_3)-\mathcal{H}(\mathbf{W}_{1,2,3}).
\end{equation*}
\end{definition}
The following corollary follows from Theorem 2 in\cite{yeung1991new}.
\begin{cor}\label{cor:interaction-information-bounds} 
The graphon interaction information satisfies the following bounds:
  \begin{align*}
    -\min \{&\mathrm{I}_{\mathbf{W}_{1,2,3}}(W^{(2)};W^{(3)}|W^{(1)}), \mathrm{I}_{\mathbf{W}_{1,2,3}}(W^{(1)};W^{(3)}|W^{(2)}), \mathrm{I}_{\mathbf{W}_{1,2,3}}(W^{(1)};W^{(2)}|W^{(3)})\} \\
    & \quad \quad \quad \quad \quad \quad \quad \leq \mathrm{I}_{\mathbf{W}_{1,2,3}}(W^{(1)};W^{(2)};W^{(3)}) \leq \\
    \min \{&\mathrm{I}_{\mathbf{W}_{1,2}}(W^{(1)};W^{(2)}), \mathrm{I}_{\mathbf{W}_{1,3}}(W^{(1)};W^{(3)}), \mathrm{I}_{\mathbf{W}_{2,3}}(W^{(2)};W^{(3)})\}
\end{align*}
\end{cor}
Combining the interpretations of synergy and redundancy form~\cite{rosas2019quantifying} and the bounds given in Corollary~\ref{cor:interaction-information-bounds} that directly follow from~\cite{yeung1991new}, we can interpret the graphon interaction information as follows: Graphon interaction information that is close to the lower bound indicates a highly synergistic interaction among the three graphs, where the combined effect of the graphs is greater than the sum of their individual effects. Conversely, graphon interaction information that approaches the upper bound suggests highly redundant interactions, where the total information provided by all three graphs is no more than the information provided by the most informative individual graph or pair of graphs. A similar interpretation can be achieved by normalizing the graphon interaction information to the range $[-1,1]$, or alternatively to the normalized versions of the bounds provided by Corollary~\ref{cor:interaction-information-bounds} using simple algebra. We will provide the upper and lower bounds whenever disucssing the implications of the interaction information.

The graphon interaction can be extended for the setting of $d$ graphs, where $d>3$ as follows
\begin{equation}
 \mathrm{I}_{\mathbf{W}_{1,\ldots,d}}(W^{(1)};\ldots;W^{(d)})=\sum_{i=1}^d \mathcal{H}(W^{(i)})-\sum_{1\leq i <j\leq n}\mathcal{H}(\mathbf{W}_{i,j})+\ldots+(-1)^d \mathcal{H}(\mathbf{W}_{1,\ldots,d}),
\end{equation}
where $\mathcal{H}_{\mathbf{W}_{1,\ldots,d}}$ denotes the joint graphon entropy among $d$ exchangeable graphs $G_1(n,W^{(1)})$, $\ldots,$ $G_d(n,W^{(d)})$.

\subsection{Graphon total correlation, dual total correlation and O--information}
We introduce the graphon total correlation for three graphs to be the relative entropy between the joint distribution among the three graphs and the product of the individual marginals.   Rényi's two desiderata~\cite{renyi1959measures} for effective dependence measures, applicable to our setting, are: \emph{a)} the measure of dependence equals 0 if the random variables are independent or conditionally independent, and it ranges between 0 and 1 in all other cases; \emph{b)} it remains unchanged with individual one-to-one transformations applied to the variables. As noted by~\cite{joe1989relative}, a natural approach to fulfill condition \emph{a)} involves defining a `distance' between a joint distribution and a distribution that represents independence or conditional independence. Relative entropy, with the necessary normalization, is also noted by~\cite{joe1989relative} to satisfy both desiderata \emph{a)} and \emph{b)}. Consequently, we utilize the relative entropy between the joint distribution and the product of the marginal distributions, as well as the conditional relative entropy. 
Unlike interaction information, total correlation (TC) and dual total correlation (DTC) are nonnegative and satisfy monotonicity. However, as noted by
~\cite{austin2018multi}, TC and DTC are members of a large family of measures that are obtained by various entropy gaps and are a result of the Shearer's inequality first published in~\cite{chung1986some}.  TC can be simplified and expressed in terms of entropies, as given in Definition~\ref{def-trivariate-tc}. It has been shown by~\cite{austin2018multi,austin2018measure} that DTC plays an important role in a new decomposition theorem for measures on product spaces.

\begin{definition}[Graphon total correlation]\label{def-trivariate-tc} We define $\mathrm{TC}_{\mathbf{W}_{1,2,3}}(W^{(1)},W^{(2)},W^{(3)})$, the graphon total correlation among three exchangeable random graphs $G_1(n,W^{(1)})$, $G_2(n,W^{(2)})$ and $G_3(n,W^{(3)})$ as follows:
    \begin{equation*}
\mathrm{TC}_{\mathbf{W}_{1,2,3}}(W^{(1)},W^{(2)},W^{(3)})=\sum_{i=1}^3 \mathcal{H}(W^{(i)})-\mathcal{H}(\mathbf{W}_{1,2,3}).
    \end{equation*}
\end{definition}

The graphon total correlation can be extended for the setting of $d$ graphs, $G_1(n,W^{(1)})$, $\ldots,$ $G_d(n,W^{(d)})$ where $d>3$ as follows
\begin{equation}\label{eq-total-correlation-dvariate}
   \mathrm{TC}_{\mathbf{W}_{1,\ldots,d}}(W^{(1)},\ldots,W^{(d)})=\sum_{i=1}^d \mathcal{H}(W^{(i)})-\mathcal{H}(\mathbf{W}_{1,\ldots,d})
\end{equation}
The term $\mathcal{H}(\mathbf{W}_{1,\ldots,d})$ represents the actual amount of information contained within the $d$ graph generating mechanisms.
The difference between these terms highlights the total redundancy among the set of graph generating mechanisms. A graphon total correlation close to zero indicates that the graphs in the system are largely statistically independent, implying that knowledge of one graph provides little or no insight into the others.
The maximum graphon total correlation for a fixed set of individual entropies $\mathcal{H}(W^{(1)}),\ldots,\mathcal{H}(W^{(d)})$ is given as
\begin{equation*}\label{eq:total-correlation-dvariate}
\mathrm{TC}^{\mathrm{max}}_{\mathbf{W}_{1,\ldots,d}}(W^{(1)},\ldots,W^{(d)})=\sum_{i=1}^d \mathcal{H}(W^{(d)})-\max_{i}\mathcal{H}(W^{(i)}).
\end{equation*}
We could normalize the graphon total correlation defined in \eqref{eq-total-correlation-dvariate} to lie in the range $[0,1]$ as follows:
\begin{align*}
\nonumber \mathrm{TC}^{\mathrm{N}}_{\mathbf{W}_{1,\ldots,d}}(W^{(1)},\ldots,W^{(d)})=\frac{\mathrm{TC}_{\mathbf{W}_{1,\ldots,d}}(W^{(1)},\ldots,W^{(d)})}{\sum_{j}\mathcal{H}(W^{(j)})-\max_j \mathcal{H}(W^{(j)})}
\end{align*}
The reason for the transformation can be seen from the bivariate version, where $\mathrm{TC}^{\mathrm{N}}_{\mathbf{W}_{1,2}}(W^{(1)},W^{(2)})$ can be written as
\begin{equation*}
    [\mathcal{H}(W^{(1)})+\mathcal{H}(W^{(2)})-\mathcal{H}(\mathbf{W}_{1,2})]/\min[\mathcal{H}(W^{(1)}),\mathcal{H}(W^{(2)}),
\end{equation*}
which is the normalized graphon mutual information between two exchangeable random graphs. The numerator can be written in terms of conditional entropy, see~\cite{cover1999elements}, and if without loss of generality (wlog) we consider that $\mathcal{H}(W^{(1)})<\mathcal{H}(W^{(2)})$, then $\mathrm{TC}^{\mathrm{N}}_{\mathbf{W}_{1,2}}(W^{(1)},W^{(2)})$ can be interpreted as the relative reduction of uncertanity in the generating mechanism $W^{(1)}$ of $G_1(n,W^{(1)})$ given the generating mechanism $W^{(2)}$ of $G_2(n,W^{(2)})$. 
\begin{definition}[Conditional graphon total correlation]\label{def-cond-tc} We define $\mathrm{TC}_{\mathbf{W}_{1,2,3}}(W^{(1)},W^{(2)}|W^{(3)})$, the conditional graphon total correlation 
 between two exchangeable random graphs $G_1(n,W^{(1)})$, $G_2(n,W^{(2)})$ given the generating mechanism of the third graph $G_3(n,W^{(3)})$ as follows:
    \begin{align*}
       &\mathrm{TC}_{\mathbf{W}_{1,2,3}}(W^{(1)},W^{(2)}|W^{(3)})=\mathcal{H}(\mathbf{W}_{1,3})+\mathcal{H}(\mathbf{W}_{2,3})-\mathcal{H}(\mathbf{W}_{1,2,3})-\mathcal{H}(W^{(3)}).
    \end{align*}
\end{definition}
Similarly, the conditional graphon total correlation can be extended to the setting of $d-$exchangeable random graphs $G_1(n,W^{(1)})$, $\ldots$, $G_d(n,W^{(d)})$ as follows:
\begin{align*}
    \mathrm{TC}_{\mathbf{W}_{1,\dots,d}}(W^{(1)},\ldots,W^{(d-1)}|W^{(d)})=\sum_{i=1}^{d-1}\mathcal{H}(\mathbf{W}_{i,d})-\mathcal{H}(\mathbf{W}_{1,\ldots,d})-(d-2)\mathcal{H}(W^{(d)}).
    \end{align*}
As specified by~\cite{joe1989relative} for the conditional relative entropy for random variables, the maximum value that the graphon dual total correlation attains thus is $\min_{i \in [d-1]} \{\mathcal{H}(\mathbf{W}_{i,d}-\mathcal{H}(W^{(i)})\}$. Dividing by this maximum bound will yield a value in the range $[0,1]$.
\begin{definition} We define $\mathrm{DTC}(\mathbf{W}_{1,\ldots,d})$, the graphon dual total correlation among $d$ exchangeable random graphs $G_1(n,W^{(1)}), \ldots,G_d(n,W^{(d)})$ as follows:
    \begin{align*}
\mathrm{DTC}(\mathbf{W}_{1,\ldots,d})=\mathcal{H}(\mathbf{W}_{1,\ldots,d})-(\sum_{i=1}^d \mathcal{H}(\mathbf{W}_{1,\ldots,d})-\mathcal{H}(\mathbf{W}_{1,\ldots,i-1,i+1,\ldots,d})),
    \end{align*}
    where $\mathcal{H}(\mathbf{W}_{1,\ldots,i-1,i+1,\ldots,d})$ denotes the joint graphon entropy among the graphs $G_1(n,W^{(1)})$,$\ldots$,\\$G_{i-1}(n,W^{(i-1)})$,$G_{i+1}(n,W^{(i+1)})$,$\ldots$,$G_d(n,W^{(d)})$, excluding $G_i(n,W^{(i)})$ for $i\in [d]$.
\end{definition}
The graphon dual total correlation, like the original dual total correlation defined by~\cite{han1978nonnegative}, is nonnegative and bounded by the joint entropy. Therefore, it can be normalized to lie in the range $[0,1]$. Both total correlation and dual total correlation are nonnegative and are zero only when variables are independent. Since TC is a KL divergence between the joint $d$-variate distribution and the product of its marginal distributions, a condition like $\mathrm{TC} = o(d)$ implies that the joint distribution of the $d$-variate graph limit model closely approximates a product measure in KL distance, as per~\cite{austin2018measure}. However, DTC itself is not directly a KL divergence, but since $\mathrm{TC} \leq (n-1) \cdot \mathrm{DTC}$, the analysis for TC can be extended to DTC when $\mathrm{DTC} = o(1)$. If $\mathrm{DTC} = o(n)$, then the joint distribution closely resembles a ``low-complexity'' mixture of product measures in KL distance.

 Given previous discussions on how interaction information reveals synergy and redundancy in systems, the $O$-information provides a measure for redundancy and synergy involving more than three random variables. Building on the work by~\cite{rosas2019quantifying}, which quantifies higher-order dependencies using $O$-information, we extend this measure to graphs and define the graphon $O$-information as follows:
\begin{definition}[Graphon O-information]  We define $\Omega_{\mathbf{W}_{1,\dots,d}}(W^{(1)},\ldots,W^{(d)})$, \\ the graphon $O-$information for $d$ exchangeable random graphs $G_1(n,W^{(1)})$, $\ldots$, $G_{d}(n,W^{(d)})$ as follows:
    \begin{align*}
   \Omega_{\mathbf{W}_{1,\dots,d}}(W^{(1)},\ldots,W^{(d)})=\mathrm{TC}_{\mathbf{W}_{1,\dots,d}}(W^{(1)},\ldots,W^{(d)})-\mathrm{DTC}_{\mathbf{W}_{1,\dots,d}}(W^{(1)},\ldots,W^{(d)}).
    \end{align*}
\end{definition}
The properties of the $O$-information discussed in~\cite{rosas2019quantifying} directly translate to the graphon $O$-information. Consequently, scenarios dominated by redundancy, where three or more graphs contain copies of the same information, as well as synergy-dominated systems, characterized by high-order patterns that cannot be discerned from lower-order marginals, are effectively distinguished by the graphon $O$-information. This metric captures the dominant characteristics of multivariate interdependency. We also reiterate that in the context of three graphs, the graphon $O$-information corresponds to graphon interaction information.

\section{Estimation of multivariate graphon information--theoretic measures}\label{sec-est-bivariate}
\subsection{Estimation of graphon mutual information}
In this section, we define an estimator for the graphon mutual information between two exchangeable random graphs, $G_1(n, W^{(1)})$ and $G_2(n, W^{(2)})$. This graphon mutual information is estimated based on the entropies of $W^{(1)}$, $W^{(2)}$, and $\mathbf{W}_{1,2}$. Building on our previous work, where we provided a consistent estimator of graphon entropy for any exchangeable random graph with a smooth graphon representative~\cite{skeja2023entropy} we now extend this methodology to estimate the joint graphon entropy. This involves computing the entropy of the estimated graphon system $\widehat{\mathbf{W}}_{1,2}$, for which we also provide its convergence rate. We define the joint graphon entropy estimator between two graphs $G_1(n,W^{(1)})$ and $G_2(n,W^{(2)})$ as follows:\begin{equation}\label{eq:bivariate-est-entropy}
    \widehat{\mathcal{H}}(\mathbf{W}_{1,2})=\mathcal{H}(\widehat{\mathbf{W}}_{1,2}),
\end{equation} where \begin{equation}\label{eq-joint-entropy-estimator}
    \widehat{\mathbf{W}}_{1,2}=[\widehat{W}^{(12)} \quad (\widehat{W}^{(1)}-\widehat{W}^{(12)}) \quad (\widehat{W}^{(1)}-\widehat{W}^{(12)}) \quad (1-\widehat{W}^{(1)}-\widehat{W}^{(2)}+\widehat{W}^{(12)})].
\end{equation} 
\subsubsection{Estimation of the multivariate graphon}\label{subsec-estimation-joint-graphon}
We will estimate $W^{(1)}(x,y)$, $W^{(2)}(x,y)$ and $W^{(12)}(x,y)$, up to re-arrangement of the axes, given the adjacency matrices $A^{(1)}$ and $A^{(2)}$, both of size $n\times n$ using a correlated two-layer stochastic block model~\cite{pamfil2020inference} with a single community vector and size $h$. The community vector $z$ of length $n$ groups nodes of the graphs that should lie on the same group, and we let $\mathcal{Z}_k$ denote all possible community vectors that respect $n=hk+r$, where $h$ denotes the community size or bandwidth, $k$ the number of communities, and $r$ is a remainder term between $0$ and $h-1$. All components of $z$ take values in $[k]$. The main difficulty is in estimating the community vector $z$. Following the approach adopted in~\cite{Olhede_2014} for the case of a single graph (see details in Appendix II ), we estimate the shared community membership vector $z$ between adjacency matrices $A^{(1)}$ and $A^{(2)}$ by the method of profile maximum likelihood as follows:

\begin{align}\label{z-est}
   \nonumber \hat{z}:=\argmax_{z \in \mathcal{Z}_k}&\sum_{i<j}\{A_{ij}^{(1)} A_{ij}^{(2)}\log (\Bar{A}_{z_i z_j}^{(12)})+A_{ij}^{(1)}(1-A_{ij}^{(2)})\log(\Bar{A}_{z_i z_j}^{(1)}-\Bar{A}_{z_i z_j}^{(12)})\\
    &+(1-A_{ij}^{(1)})A_{ij}^{(2)}\log(\Bar{A}_{z_i z_j}^{(2)}-\Bar{A}_{z_i z_j}^{(12)})+(1-A_{ij}^{(1)})(1-A_{ij}^{(2)})\log(1-\Bar{A}_{z_i z_j}^{(1)}-\Bar{A}_{z_i z_j}^{(2)}+\Bar{A}_{z_i z_j}^{(12)})\},
\end{align}
where for all $1\leq a,b \leq k$,
\begin{align}\label{theta-est}
&\Bar{A}_{ab}^{(d)}=\frac{A^{(d)}_{ij}I(\hat{z}_i=a)I(\hat{z}_j=b)}{I(\hat{z}_i=a)I(\hat{z}_j=b)},
\end{align}
for $(d)=(1),(2),(12)$, and $A^{(12)}$ is as given in Equation~\eqref{A-12}. See Appendix II for details of the estimation procedure. Since all adjacency matrices are symmetric, we have $\Bar{A}^{(d)}_{ab}=\Bar{A}^{(d)}_{ba}$.
Combining \eqref{z-est}-\eqref{theta-est}, and following~\cite{Olhede_2014}, we define the estimators of $W^{{(1)}}$, $W^{{(2)}}$ and $W^{{(12)}}$ as follows
\begin{align}
    &\widehat{W}^{(1)}(x,y;h):=\Bar{A}^{(1)}_{\min(\lceil nx/h\rceil,k) \min(\lceil ny/h\rceil,k)}, \quad 0<x,y<1,\\
    &\widehat{W}^{(2)}(x,y;h):=\Bar{A}^{(2)}_{\min(\lceil nx/h\rceil,k) \min(\lceil ny/h\rceil,k)}, \quad 0<x,y<1,\\
    &\widehat{W}^{(12)}(x,y;h):=\Bar{A}^{(12)}_{\min(\lceil nx/h\rceil,k) \min(\lceil ny/h\rceil,k)}, \quad 0<x,y<1.
\end{align}
 Next, we can adopt the rate-optimal graphon estimation result from~\cite{gao2015rate,klopp2017oracle} for each of $\widehat{W}^{(1)}$, $\widehat{W}^{(2)}$, and $\widehat{W}^{(12)}$ since all  $A^{(1)}$, $A^{(2)}$, and $A^{(12)}$ are valid adjacency matrices.

\begin{cor}\label{cor-joint-graphon}
 Assume $A^{(d)}$ is generated from a $\alpha-$H\"older graphon $W^{(d)}\in \mathcal{F}^{\alpha}(M)$ for $(d)=(1),(2),(12)$ following Definition~\ref{assump:smooth}. Then, the following inequality holds with probability at least $1-\exp(-C'n)$:
\begin{align*}
    \inf_{\tau \in \mathcal{M}}\iint_{(0,1)^2}|W^{(d)}\left(\tau(x),\tau(y)\right)-\widehat{W}^{(d)}(x,y)|^2 \ dx dy \leq C\left(n^{-2\alpha/(\alpha+1)}+\frac{ \log n}{n}+\frac{1}{n^{\alpha}}\right), 
\end{align*}
where $\mathcal{M}$ is the set of all measure--preserving bijections $\tau:[0,1]\to [0,1]$, $C'>0$ only depends on $M$ and the universal constant $C>0$.
\end{cor} 
As previously discussed in Section~\ref{subsec-estimation-joint-graphon} and Corollary~\ref{cor-joint-graphon}, $\widehat{W}^{(1)}$, $\widehat{W}^{(2)}$ and $\widehat{W}^{(12)}$ can be consistently estimated, thus yielding to a consistent estimator of $\mathbf{W}_{1,2}$. Once we have estimated $\widehat{W}^{(1)}$, $\widehat{W}^{(2)}$ and $\widehat{W}^{(12)}$ from realizations of $A^{(1)}$, $A^{(2)}$ and $A^{(12)}$, respectively, we can form the joint graphon entropy estimator using Equations~\eqref{eq:bivariate-est-entropy}--\eqref{eq-joint-entropy-estimator} and Definition~\ref{def-joint-graphon-entropy}.
\begin{theorem}\label{thm-conv-rate} Let $\mathbf{W}_{1,2}$ denote the system of graphons as defined in~\eqref{eq-system-graphons-bivariate}, and let $W^{(1)}, W^{(2)},$ and $W^{(12)}$ be functions in $\mathcal{F}^{\alpha}(M)$, where each function is $\alpha$-Hölder continuous. Then, the following inequality holds with probability at least $1-\exp(-C'n)$:

    \begin{equation*}
        |\widehat{\mathcal{H}}(\mathbf{W}_{1,2})-\mathcal{H}(\mathbf{W}_{1,2})|\leq  C\left(n^{-2\alpha/(\alpha+1)}+\frac{\log n}{n}+\frac{1}{n^{\alpha}}\right)^{1/2},
    \end{equation*}
 where $C'>0$ only depends on $M$ and the universal constant $C>0$.
\end{theorem}
\begin{proof}
    See Appendix I. B.
\end{proof}
Building on the joint graphon entropy estimator, the graphon mutual information estimator naturally follows. We define $\widehat{\mathrm{I}}_{\mathbf{W}_{1,2}}(W^{(1)}; W^{(2)})$ as an estimator for the mutual information of the bivariate graph limit model. This is given by $\widehat{\mathrm{I}}_{\mathbf{W}_{1,2}}(W_1;W_2) = \widehat{\mathcal{H}}(W^{(1)}) + \widehat{\mathcal{H}}(W^{(2)}) - \widehat{\mathcal{H}}(\mathbf{W}_{1,2})$, where $\widehat{\mathcal{H}}(W^{(i)})$ is the graphon entropy estimator for $W^{(i)}$, $i=1,2$.
\begin{cor}\label{cor:mi-conv-rate} Assume the conditions of Theorem \ref{thm-conv-rate} hold.  Then, the following inequality holds with probability at least $1-\exp(-C'n)$:
    \begin{align*}
        |\widehat{\mathrm{I}}_{\mathbf{W}_{1,2}}(W_1;W_2)-\mathrm{I}_{\mathbf{W}_{1,2}}(W_1;W_2)|\leq  C\left(n^{-2\alpha/(\alpha+1)}+\frac{\log n}{n}+\frac{1}{n^{\alpha}}\right)^{1/2},
    \end{align*}
  where $C'>0$ only depends on $M$ and the universal constant $C>0$.
\end{cor}
\begin{proof}
    Combining Theorem~\ref{thm-conv-rate} given above and Corollary 2 from~\cite{skeja2023entropy} gives the desired result. 
\end{proof}
Corollary~\ref{cor:mi-conv-rate} follows directly from Theorem~\ref{thm-conv-rate}, as the convergence rate of the joint graphon entropy is the leading term. It demonstrates that the joint graphon entropy, and thus the graphon mutual information estimator, converge to the true values as the number of nodes $n$ increases. This convergence occurs with a probability of at least $1-\exp(-C'n)$, where $C'$ is as defined previously.

\subsection{Estimation of trivariate graphon information--theoretic measures}
In this subsection, we form estimators for the graphon total correlation, graphon interaction information, graphon dual total correlation, and their conditional versions across three exchangeable random graphs: $G_1(n, W^{(1)})$, $G_2(n, W^{(2)})$, and $G_3(n, W^{(3)})$. These estimators involve a combination of graphon entropy estimators and joint graphon entropy estimators for pairs and triplets of graphs. Consequently, all of them share a similar estimation method and converge at the same rate. We have previously defined the joint graphon entropy estimator for two graphs in the preceding subsection, and the graphon entropy estimator has been detailed in our prior work~\cite{skeja2023entropy}. We now proceed to define the joint graphon entropy estimator for three graphs as follows:

\begin{equation}\label{joint-trivariate-entropy}
    \widehat{\mathcal{H}}(\mathbf{W}_{1,2,3})=\mathcal{H}(\widehat{\mathbf{W}}_{1,2,3})
\end{equation} where \begin{align}\label{eq-joint-entropy-estimato}
  \nonumber  \widehat{\mathbf{W}}_{1,2,3} =\Big[&\widehat{W}^{(123)} \quad (\widehat{W}^{(12)}-\widehat{W}^{(123)}) \quad (\widehat{W}^{(13)}-\widehat{W}^{(123)}) \quad (\widehat{W}^{(23)}-\widehat{W}^{(123)}) \\
  \nonumber  &(\widehat{W}^{(1)}-\widehat{W}^{(12)}-\widehat{W}^{(13)}+\widehat{W}^{(123)})\quad  (\widehat{W}^{(2)}-\widehat{W}^{(12)}-\widehat{W}^{(23)}+\widehat{W}^{(123)}) \\
   &(\widehat{W}^{(3)}-\widehat{W}^{(13)}-\widehat{W}^{(23)}+\widehat{W}^{(123)}) \\
 \nonumber  & (1-\widehat{W}^{(1)}-\widehat{W}^{(2)}+\widehat{W}^{(3)}+\widehat{W}^{(12)}+\widehat{W}^{(13)}+\widehat{W}^{(23)}-\widehat{W}^{(123)})\Big].
\end{align}

Following Section~\ref{subsec-estimation-joint-graphon}, we analogously define estimators for $W^{(1)}, W^{(2)}, W^{(3)}, W^{(12)}, W^{(13)}, W^{(23)}$, and $W^{(123)}$ given adjacency matrices $A^{(1)}, A^{(2)}$ and $A^{(3)}$ of size $n\times n$, latent vector $\xi$ and a single common bandwidth $h\in \mathbb{Z}^+$, using a correlated three-layer stochastic block model. The estimators of the first six graphons, namely, the three marginal graphons $W^{(1)}, W^{(2)}, W^{(3)}$ and the three pairwise combinations $W^{(12)}, W^{(13)}, W^{(23)}$, follow directly from Section~\ref{subsec-estimation-joint-graphon}. We define the estimator of the last graphon $W^{(123)}$ as follows:
\begin{equation*}
    \widehat{W}^{(123)}(x,y;h):=\Bar{A}^{(123)}_{\min(\lceil nx/h\rceil,k)\min(\lceil ny/h\rceil,k)}, \ 0<x,y<1,
\end{equation*}
where $$\Bar{A}^{(123)}_{ab}=\frac{A^{(123)}_{ij}\mathbbm{1}(\hat{z}_i=a)\mathbbm{1}(\hat{z}_j=b)}{\mathbbm{1}(\hat{z}_i=a)\mathbbm{1}(\hat{z}_j=b)},$$
for $1\leq a,b \leq k$, where $k=\lceil n^{1/[(\alpha \wedge 1)+1]}\rceil$ and
$A^{(123)}_{ij}=A^{(1)}_{ij} A^{(2)}_{ij}  A^{(3)}_{ij}$, is also an adjacency matrix. Corollary~\ref{cor-joint-graphon} applies to the estimation of trivariate joint graphon as well, where in this case now $(d)$ assumes values $(1)$, $(2)$, $(3)$, $(12)$, $(13)$, $(23)$, $(123)$. Consequently, a corollary of Theorem~\ref{thm-conv-rate} is the following:
\begin{cor}\label{cor:trivariate-rate} Let $\mathbf{W}_{1,2,3}$ denote the system of graphons as defined in~\eqref{eq-system-graphon-trivariate}, and let each $W^{(i)}, W^{(ij)},$ and $W^{(ijk)}$ for $1\leq i,j,k \leq 3$ be functions in $\mathcal{F}^{\alpha}(M)$, where each function is $\alpha$-Hölder continuous. Then, the following inequality holds with probability at least $1-\exp(-C'n)$:
    \begin{align*}
        &|\widehat{\mathcal{H}}(\mathbf{W}_{1,2,3})-\mathcal{H}(\mathbf{W}_{1,2,3})|\leq  C\left(n^{-2\alpha/(\alpha+1)}+\frac{\log n}{n}+\frac{1}{n^{\alpha}}\right)^{1/2},
    \end{align*}
  where $C'>0$ only depends on $M$ and the universal constant $C>0$. 
\end{cor}
The proof of Corollary~\ref{cor:trivariate-rate} follows analogously from the proof of Theorem~\ref{thm-conv-rate}. Given that both the bivariate joint graphon entropy estimator and the trivariate joint graphon entropy estimator satisfy the same convergence rate, the convergence rates of all other multivariate graphon information--theoretic measures are the same. This estimation method can similarly be extended to $d-$variate joint graphon estimation and, consequently, to the estimation of $d-$variate information-theoretic measures. However, the number of parameters required to describe a multivariate set of graphs without any imposed dependence assumptions increases geometrically with the number of graphs. Consequently, while estimation remains feasible, it comes at the cost of a geometrically increasing number of parameters.

Drawing analogies with nonparametric mutual information estimation via density estimation leads to an important consideration: Is it necessary for the bandwidth to differ from that used in graphon estimation? This issue has been addressed in ~\cite{skeja2023entropy} on graphon entropy estimation, which shows that a smaller bandwidth $h$ helps reduce bias. However, unlike in kernel density estimation where bandwidth can be incrementally reduced ~\cite{paninski2008undersmoothed}, as shown in~\cite{skeja2023entropy} such an approach is not feasible in graphon estimation. This is because the fundamental step of clustering in estimating graphons with block models cannot accommodate arbitrarily many clusters or excessively small bandwidths, as observed in ~\cite{choi2012stochastic}. It has been shown by~\cite{skeja2023entropy} that for graphon entropy estimation, the automatic bandwidth selection can be the same as that of graphon estimation proposed by~\cite{Olhede_2014}. Therefore, the automatic bandwidth selection for joint graphon estimation, chosen to be the mean bandwidth of the individual adjacency matrices, performs well. We elaborate on the performance of our estimators in the following section.
\section{Simulation Study}\label{sec-simulation}
\subsection{Performance of the estimators of the multivariate graphon information--theoretic measures}
In this section, we first assess the performance of the graphon mutual information estimator on synthetic graphs by evaluating the root mean squared error (RMSE) between the true graphon mutual information and the estimated graphon mutual information. We observe a notable trend: as the number of nodes in the graph increases, the discrepancy between the estimated graphon mutual information and its true value progressively diminishes. This trend is consistent with theoretical expectations, underscoring the estimator's improved accuracy as graph sizes increase. Furthermore, we discuss the implications of graphon mutual information density matrix in Appendix III.A. To begin, we generate adjacency matrices using various graph limit functions. For estimation purposes, we employ a two-layered correlated stochastic block model, extending the univariate scenario introduced in~\cite{Olhede_2014}. Details of the implementation are provided in Section~\ref{subsec-estimation-joint-graphon} and Appendix II.

 We consider generating $A^{(2)}_{ij}$ from $A^{(1)}_{ij}$ as described by~\cite{chandna2022edge}:
\begin{equation*}
A^{(2)}_{ij}|\xi,A^{(1)}_{ij} \sim \mathrm{Bernoulli}(h(\xi_i,\xi_j)A_{ij}^{(1)}), \quad 1\leq j<i\leq n,
    \end{equation*}
   where we choose $h(x,y)= \frac{x+y}{2y}$ for $x,y \in (0,1)$. There exists an input-output relationship, ensuring that the second adjacency matrix is a subset of the first adjacency matrix. Therefore, this method of generation is particularly useful because the joint graphon $W^{(12)}$ will be equal to $W^{(2)}$, thus the true mutual information can be easily calculated.  We examine two scenarios to evaluate the performance of the graphon mutual information estimator. In the first scenario, we set $W^{(1)}(x,y)=xy$ and $W^{(2)}=x(x+y)/2$. Conversely, in the second scenario, we consider $W^{(1)}=xy$ and $W^{(2)}=x^2y^2$, with $h(x,y)=xy$. 
   
   Figure 1 illustrates the diminishing behavior of the RMSE for increasing $n$ on a log-log scale for the estimated graphon mutual information and the true mutual information of the two graphon combinations described above. The RMSE for the first scenario is depicted by blue circles in Figure~\ref{fig:RMSEplot1}, while the RMSE for the second scenario is represented by red crosses in the same figure.
\begin{figure}[!htb]
    \centering
    \begin{minipage}{0.49\textwidth}
        \centering \includegraphics[width=0.9\linewidth]{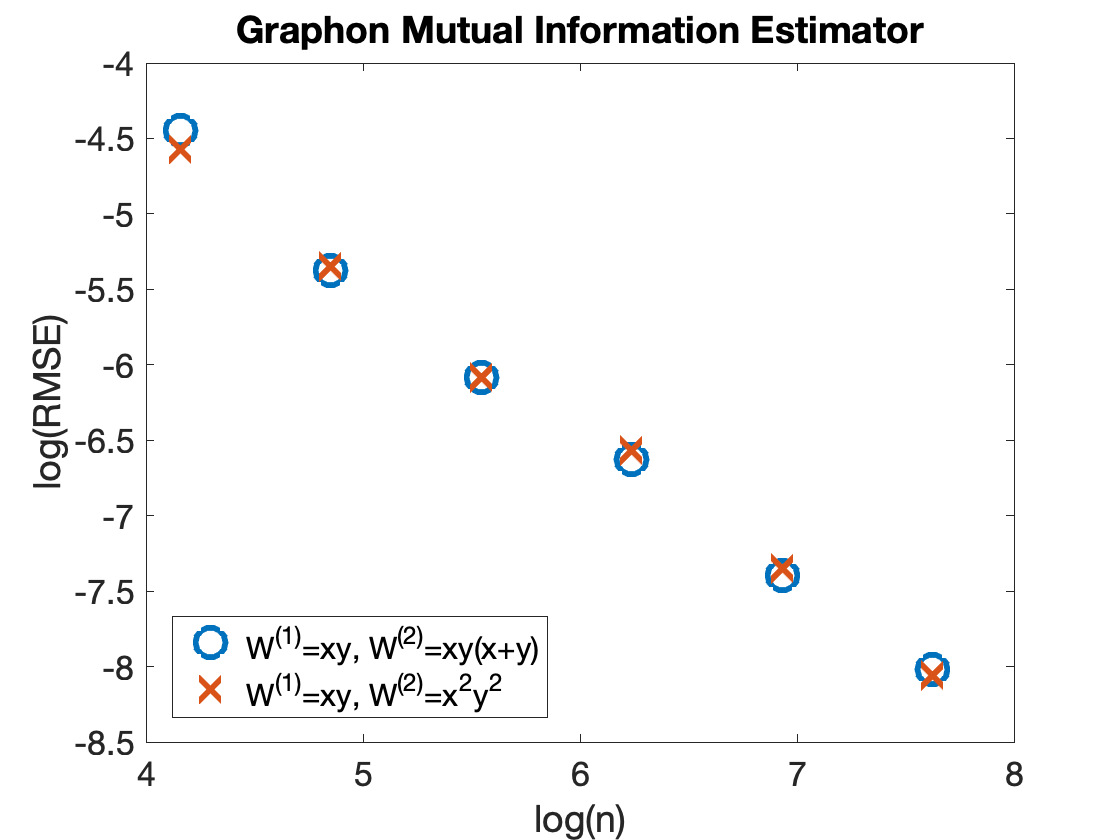}
        \caption{Log-log plots of the decay of $\mathrm{RMSE}(\widehat{\mathrm{I}}_{\mathbf{W}_{1,2}}(W_1;W_2),\mathrm{I}_{\mathbf{W}_{1,2}}(W_1;W_2))$ as a function of $n \in \{2^6,2^7,\cdots,2^{11}\}$, averaged over $300$ trials.}
        \label{fig:RMSEplot1}
    \end{minipage}\hfill
    \begin{minipage}{0.49\textwidth}
        \centering
\includegraphics[width=0.9\linewidth]{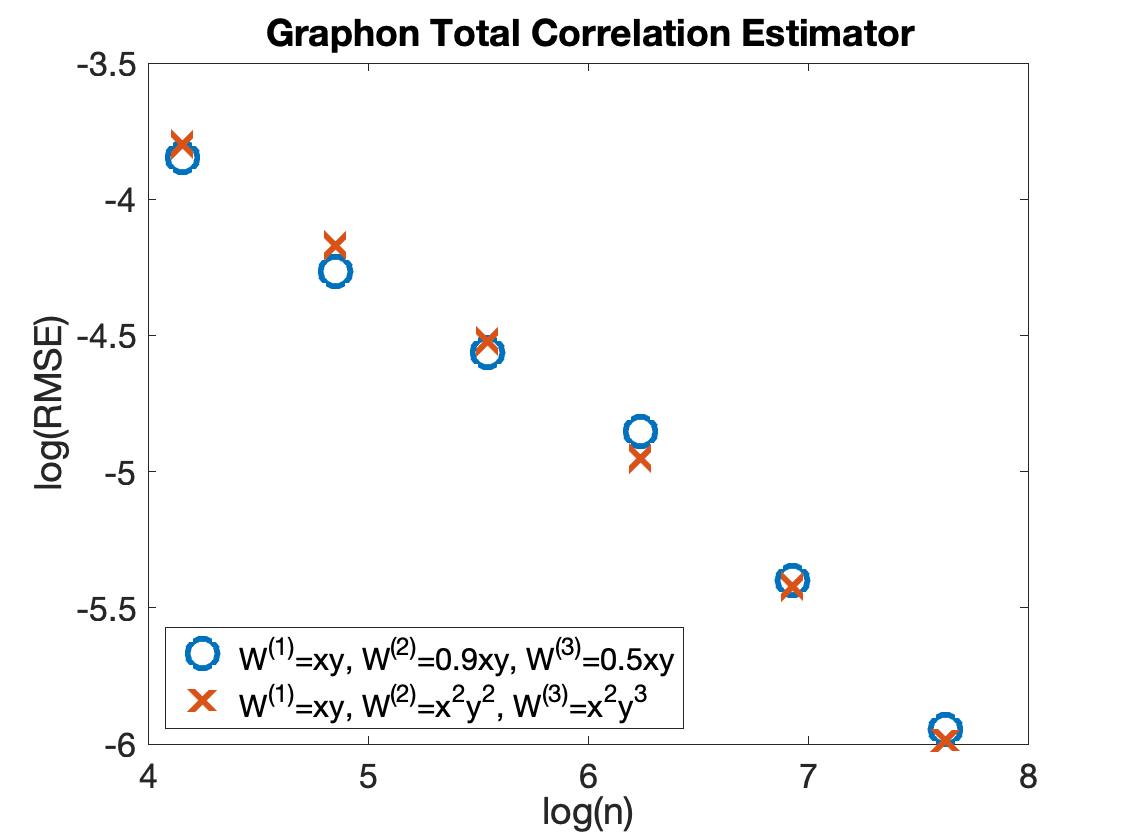} 
        \caption{Log-log plots of the decay of $\mathrm{RMSE}(\widehat{\mathrm{TC}}_{\mathbf{W}_{1,2,3}}(W^{(1)},W^{(2)},W^{(3)}),$ $\mathrm{TC}_{\mathbf{W}_{1,2,3}}(W^{(1)},W^{(2)},W^{(3)})$ as a function of $n \in \{2^6,2^7,\cdots,2^{11}\}$, averaged over $300$ trials.}
        \label{fig:RMSEplot2}
    \end{minipage}
\end{figure}

\begin{figure}[!htb]
    \centering
    \begin{minipage}{0.49\textwidth}
        \centering \includegraphics[width=0.9\linewidth]{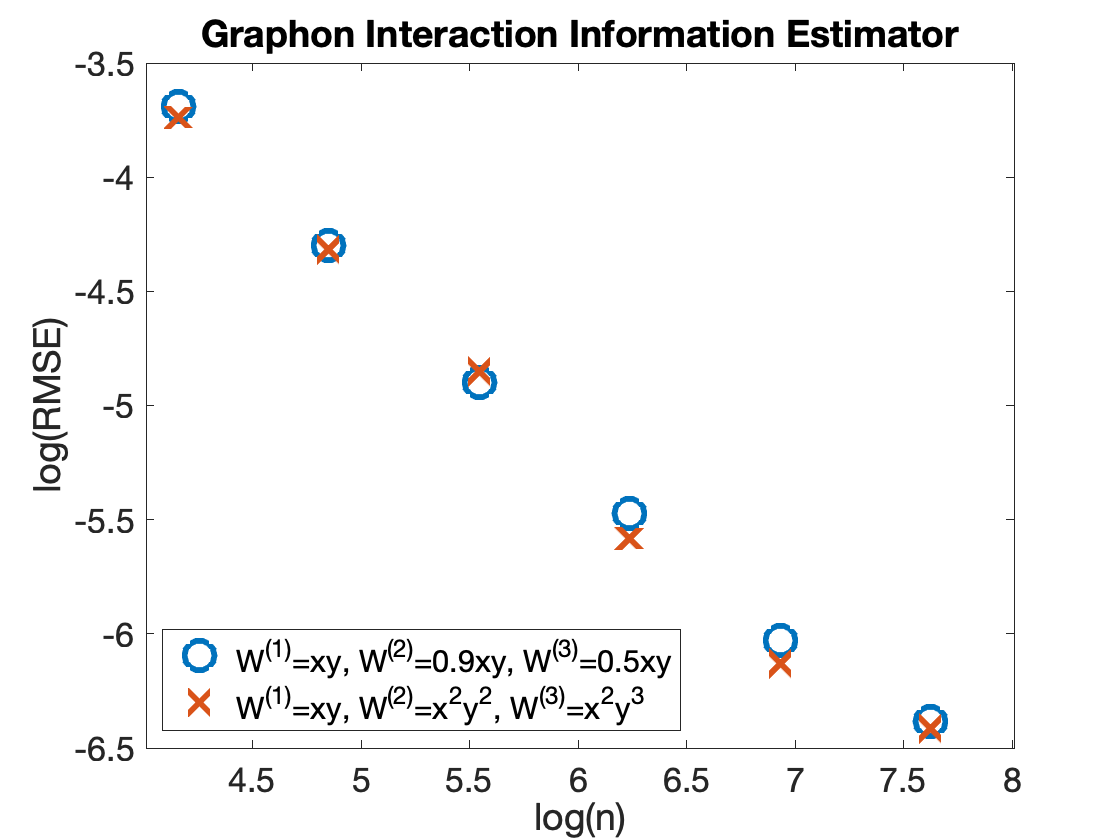}
        \caption{Log-log plots of the decay of $\mathrm{RMSE}(\widehat{\mathrm{I}}_{\mathbf{W}_{1,2,3}}(W^{(1)};W^{(2)};W^{(3)})$, $\mathrm{I}_{\mathbf{W}_{1,2,3}}(W^{(1)};W^{(2)};W^{(3)}))$ as a function of $n \in \{2^6,2^7,\cdots,2^{11}\}$, averaged over $300$ trials.}
        \label{fig:trivariate-mi}
    \end{minipage}\hfill
    \begin{minipage}{0.49\textwidth}
        \centering
\includegraphics[width=0.9\linewidth]{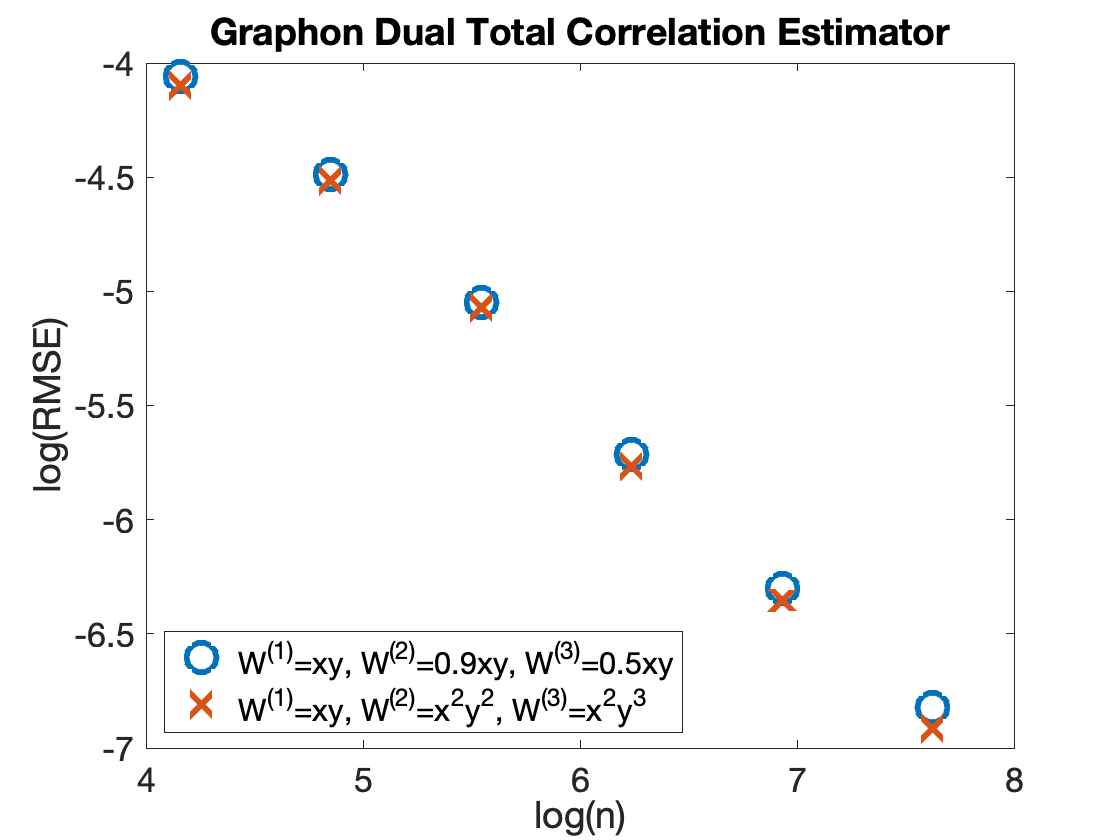} 
         \caption{Log-log plots of the decay of $\mathrm{RMSE}(\widehat{\mathrm{DTC}} _{\mathbf{W}_{1,2,3}}(W^{(1)},W^{(2)},W^{(3)}),$ $\mathrm{DTC}_{\mathbf{W}_{1,2,3}}(W^{(1)},W^{(2)},W^{(3)})$ as a function of $n \in \{2^6,2^7,\cdots,2^{11}\}$, averaged over $300$ trials.}
        \label{fig:conditional-trivariate-mi}
    \end{minipage}
\end{figure}
Building on the bivariate case, we explore the root mean squared error of the graphon total correlation, graphon interaction information, and graphon dual total correlation estimators for three graphs generated by the following two cases. The generation method is akin to the bivariate case, wherein we generate graphs that are subsets of each other in descending order. This approach ensures that $W^{(12)}=W^{(2)}$, $W^{(13)}=W^{(3)}=W^{(23)}=W^{(123)}$. In the first case, we utilize the following three graphons: $W^{(1)}=xy$, $W^{(2)}=0.8xy$, and $W^{(3)}=0.5xy$. For the second case, we use $W^{(1)}=xy$, $W^{(2)}=x^2y^2$, and $W^{(3)}=x^2y^3$. We estimate the graphon total correlation among the the three exchangeable random graphs $G_1(n,W^{(1)})$, $G_2(n,W^{(2)})$, and $G_3(n,W^{(3)})$ for both cases, using Equations \ref{joint-trivariate-entropy} and \ref{eq-joint-entropy-estimator}, as well as the univariate graphon entropy estimation method described in~\cite{skeja2023entropy}. Figure~\ref{fig:RMSEplot2} illustrates the diminishing root mean squared error with increasing $n$, consistent with our theoretical expectations. The RMSE for the first case is denoted by blue circles, while for the second case, it is represented by red crosses. 

Similarly to the estimation of graphon total correlation, estimating the graphon interaction information and the graphon dual total correlation requires estimating individual and joint graphon entropies among graph generating mechanisms. One would expect similar behavior from these estimators. However, for clarity, we illustrate the performance of the graphon interaction information estimator for two scenarios with different graph limits in Figure~\ref{fig:trivariate-mi}, showing the diminishing behavior of the root mean squared error over increasing $n$ across 300 averaged trials. Likewise, Figure~\ref{fig:conditional-trivariate-mi} illustrates the diminishing root mean squared error of the graphon dual total correlation estimator for both scenarios as $n$ increases over 300 averaged trials. Consequently, as expected, the performance of all estimators follows a similar behavior.
\subsection{Multivariate graphon information--theoretic measures and multiplex graph structure}
In the following subsections, we explore the concepts of synergy and redundancy in graphs, examining their relationships with the introduced graphon multivariate information--theoretic measures such as graphon interaction information and graphon total correlation. As discussed in Section~\ref{sec-info-measure-three}, redundancy is typically associated with positive interaction information, whereas synergy is associated with negative values. We adapt this interpretation and visually demonstrate examples of redundancy and synergy in graphs through Figures~\ref{fig:tc-sbm-percolated}--\ref{fig:case2-synergy}, respectively.

\subsubsection{Redundancy and multiplex graph structure} \label{subsubsec:redundancy and multiplex graph structure}
Figure~\ref{fig:tc-sbm-percolated} presents three graphs exhibiting a shared community structure. The second and third graphs are percolated versions of the first graph, retaining $95\%$ and $90\%$ of the edges, respectively. The normalized estimated graphon total correlation is given as $0.869$. The estimated interaction information, given as $0.285$, is very close to its upper bound, which is $0.289$. This suggests high redundancy. These values are intuitive, considering that the second and third graphs are derived from the first by removing some edges. Consistent with the method of graph generation, the estimated graphon total correlation and the estimated graphon interaction information suggest a predominant information overlap. Figure~\ref{fig:tc-sbm-percolated} further illustrates this by coloring the shared edges among the three graphs in orange, and the unshared ones in blue. 

In contrast, Figure~\ref{fig:tc-sbm-different-thetas} illustrates three graphs generated with varying community structures. In this case, the shared edges are significantly less apparent compared to Figure~\ref{fig:tc-sbm-percolated}, and as expected, the majority of the edges are not common among the three graphs. Nevertheless, there is still a portion of shared edges, indicating some redundancy. The normalized estimated graphon total correlation among these three graphs is $0.370$. The estimated graphon interaction information is $0.097$, with the lower and upper bounds given as $-0.0072$ and $0.2332$, once more indicating the presence of some redundancy within the system of the three graphs.
\begin{figure}[ht!]
    \centering
    \begin{minipage}{0.48\linewidth}
        \centering
        \includegraphics[width=\linewidth]{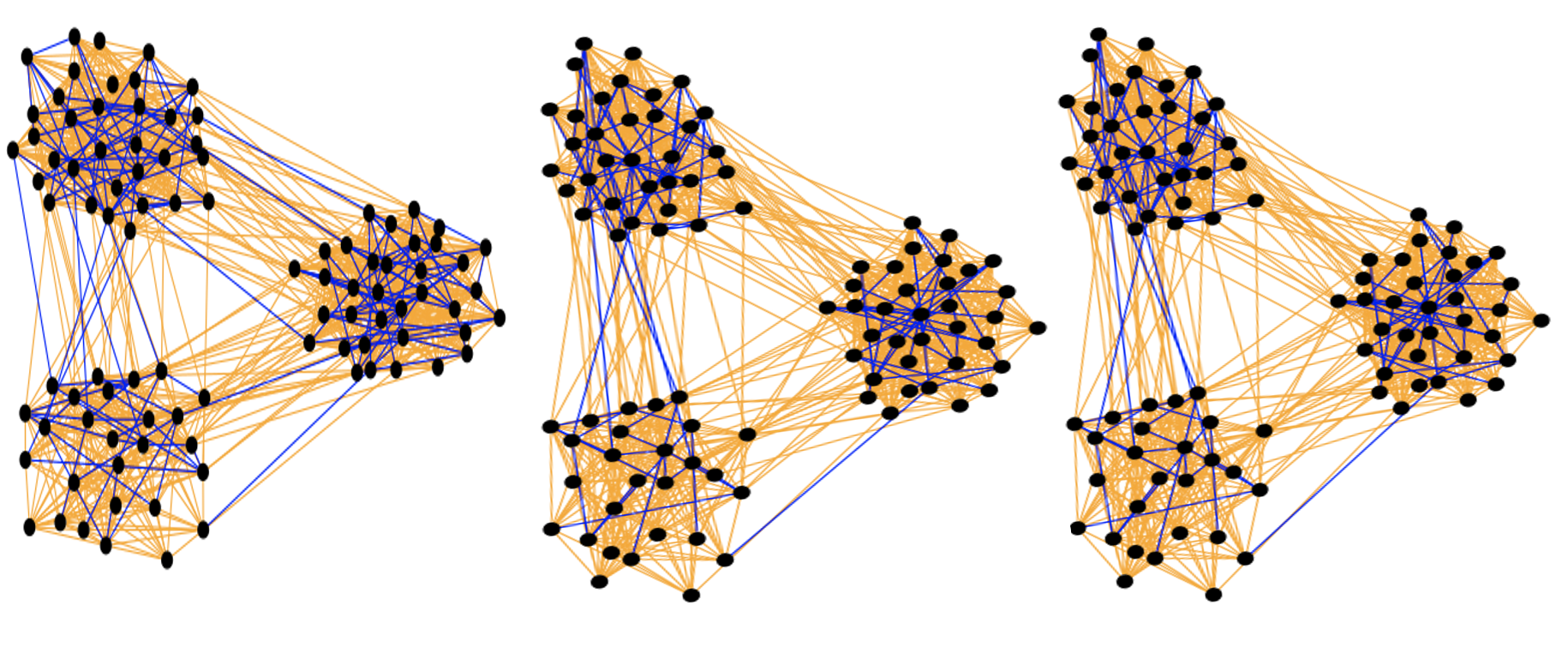}
        \caption{A parent graph (left) alongside two percolated versions (middle and right) with $5\%$ and $10\%$ of edges removed, respectively. Edges colored orange represent those shared among the graphs, while blue unshared.}
        \label{fig:tc-sbm-percolated}
    \end{minipage}
    \hfill 
    \begin{minipage}{0.48\linewidth}
        \centering
        \includegraphics[width=\linewidth]{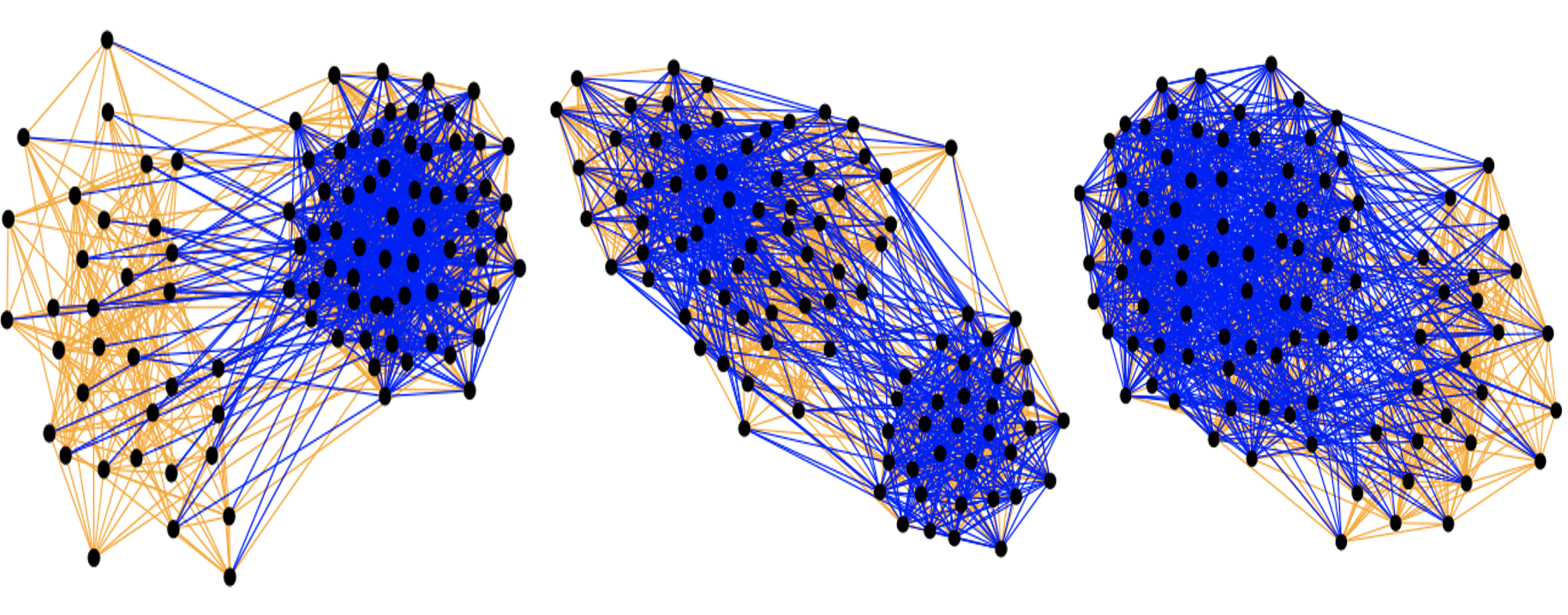}
        \caption{Three graphs generated with varying block structures.  Edges colored orange represent those shared among the graphs, while blue unshared.}
        \label{fig:tc-sbm-different-thetas}
    \end{minipage}
\end{figure}
\subsubsection{Synergy and multiplex graph structure}
Figure~\ref{fig:case1-synergy} illustrates three graphs, two of which are generated with different community structures, while the third is obtained via the pairwise \texttt{xor} operation on the adjacency matrices of the first two graphs, i.e., $A^{(3)}_{ij} = (A^{(1)}_{ij} + A^{(2)}_{ij}) \mod 2$, partly inspired by examples in~\cite{rosas2019quantifying} concerning random variables. The edges of the third graph are colored in green, and edges shared between the first two graphs and the third are also colored in green, while the rest is colored in purple. As observed from Figure~\ref{fig:case1-synergy}, the first two graphs provide little to no information about the third graph. This is further validated by estimating and then normalizing the graphon mutual information between the first and the third, and the second and the third graphs, following Equation~\ref{normalized-mi}, yielding $\widehat{\mathrm{I}}^{\mathrm{N}}_{\mathbf{W}_{1,3}}(W^{1};W^{3})=0.117$ and $\widehat{\mathrm{I}}^{\mathrm{N}}_{\mathbf{W}_{2,3}}(W^{2};W^{3})=0.156$, respectively. However, their \texttt{xor} combination precisely defines the third graph, hence the interaction information among these three graphs is given as $-0.17$. The lower and upper bounds of this interaction information, following Corollary~\ref{cor:interaction-information-bounds}, are given to be $-0.203$ and $0.07$, respectively. This results in a \emph{synergistic} behavior, as lower-order interactions (marginal and pairwise) fail to capture the higher-order interactions within the system of these three graphs.

\begin{figure}[ht!]
    \centering
    \begin{minipage}{0.48\linewidth}
        \centering
        \includegraphics[width=\linewidth]{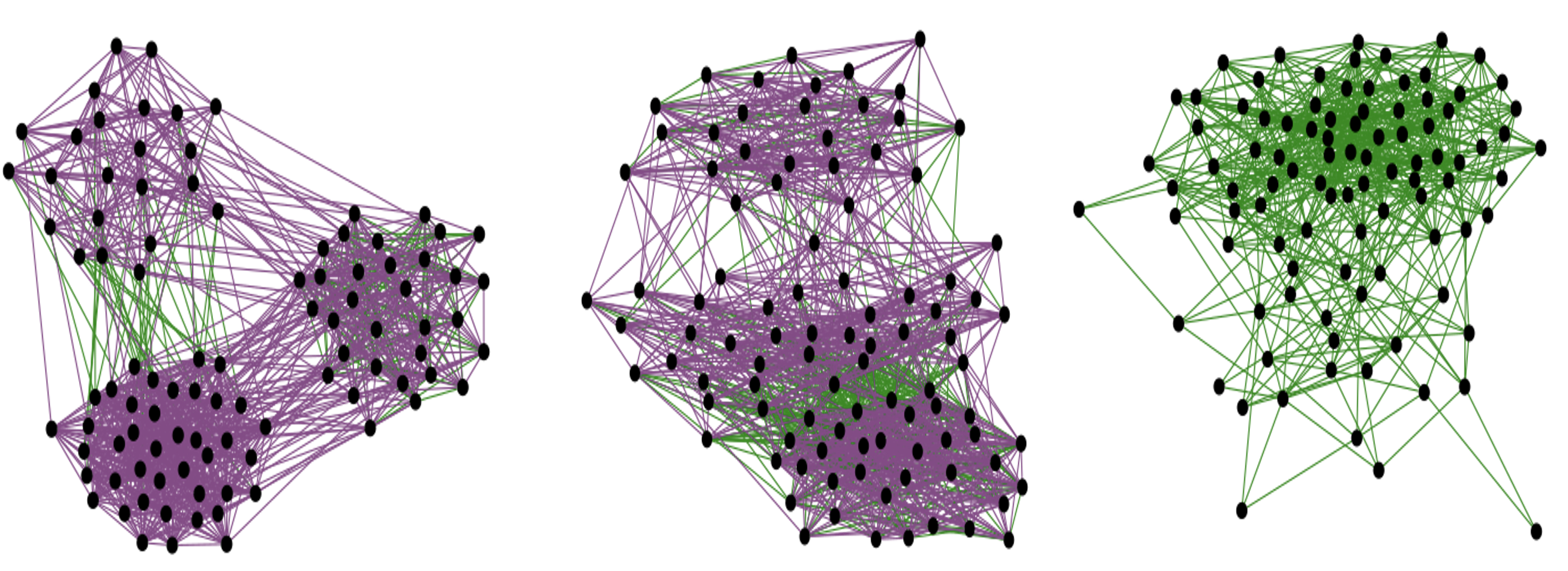}
        \caption{Two graphs generated with varying community structures (left and middle), where the third (right) graph is derived by the \texttt{xor} operation on the first two. Edges of the third graph and any edges shared with the first two graphs are colored in green, while edges present only in the first two graphs and not shared with the third are colored in purple.}
        \label{fig:case1-synergy}
    \end{minipage}
    \hfill 
    \begin{minipage}{0.48\linewidth}
        \centering
        \includegraphics[width=\linewidth]{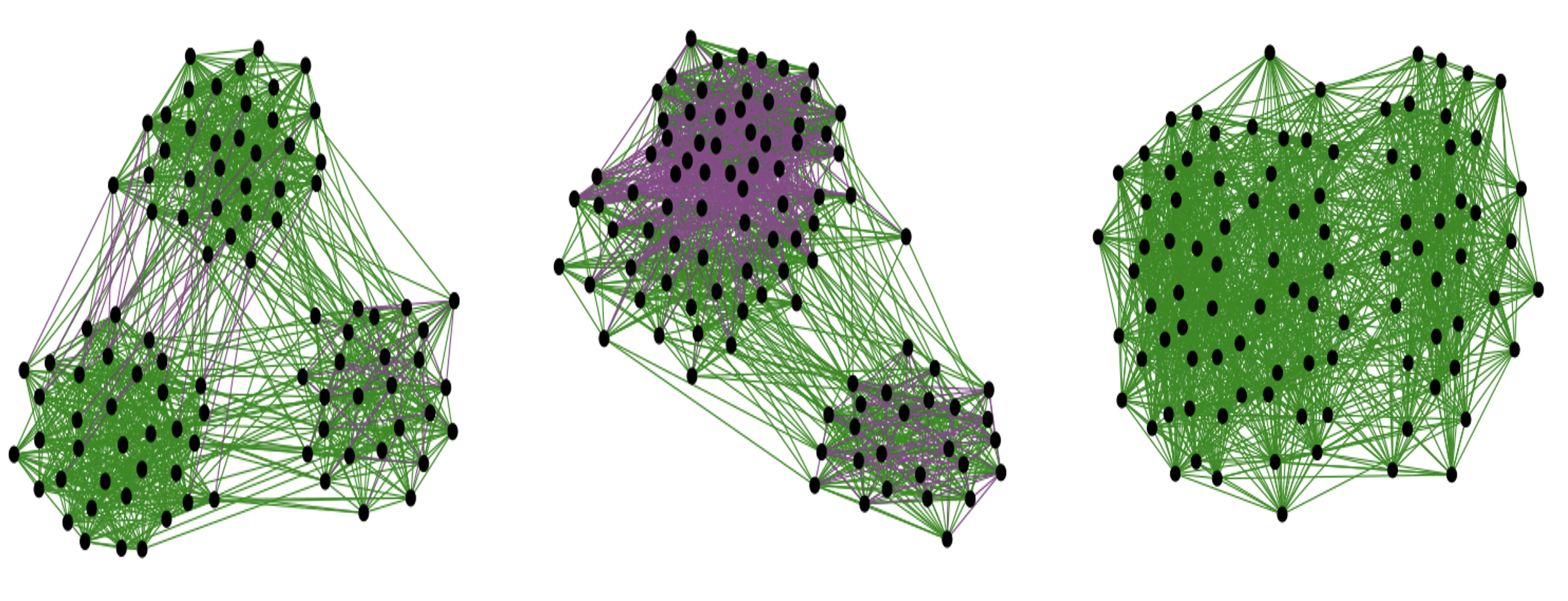}
        \caption{Three graphs generated with varying block structures. Edges of the third graph (right) and any edges shared with the first two graphs (left and middle) are colored in green, while edges present only in the first two graphs and not shared with the third are colored in purple. }
        \label{fig:case2-synergy}
    \end{minipage}
\end{figure}
Figure~\ref{fig:case2-synergy} illustrates three graphs with slightly varying community structures, where none of the graphs is a combination of the others. Similar to Figure~\ref{fig:case1-synergy}, the edges of the third graph are colored in green, and edges shared between the first two graphs and the third are also colored in green, while the rest is colored in purple. Given the shared community structure with slight variations, unlike the scenario in Figure~\ref{fig:case1-synergy}, we do not observe synergy here. As discussed in the previous subsection, this scenario exhibits redundancy within the overall system. The estimated graphon interaction information is given as $0.114$, and it is close to the upper bound, where the bounds set by Corollary~\ref{cor:interaction-information-bounds}, are $-0.010$ and $0.124$, respectively. The order of graphs in Figures~\ref{fig:case1-synergy} and~\ref{fig:case2-synergy} does not signify any particular meaning; these illustrations differentiate between synergy and redundancy, highlighting cases where pairwise shared information does not capture higher-order dependencies within a system of graphs.

\section{Real data}\label{sec-real-data}
Beyond simulations for synthetic graphs, we are also interested in understanding the insights provided by graphon mutual information, total correlation, conditional total correlation, and interaction information when applied to real-world networks. To achieve this, we analyze a temporal network that depicts interactions among the same high school students over five consecutive days, with data collected in 20-second intervals. This makes the temporal network also multiplex. The dataset is publicly available at \footnote{\url{https://networks.skewed.de/net/sp_colocation}}. We aggregate the data into hourly snapshots, resulting in a total of fifty hours. The data collection and description are detailed in~\cite{mastrandrea2015contact}.
\subsection{Graphon mutual information}
First, we construct a normalized graphon mutual information matrix of dimension $50 \times 50$ following Definition~\ref{def-mim}. This involves estimating the graphon mutual information between adjacency matrices that represent hourly interactions, and subsequently normalizing these values in accordance with Equation~\ref{normalized-mi}. A consistent pattern is observed between two specific hours of the day throughout the week, as detailed in Figure~\ref{real-data}. This figure displays selected entries from the normalized mutual information matrix, highlighting the lunch hours over four days, with the graphon mutual information values between lunch hours color-coded for clarity.

The authors of~\cite{mastrandrea2015contact} observe a variable pattern of interaction between students during class hours and a strong pattern of interaction during the two-hour lunch break. They also note that the daily activity of students exhibits a repetitive pattern throughout the week. Despite daily variations, they highlight that the temporal oscillations and mixing patterns between classes remain consistent across different days. In alignment with these observations, Figure~\ref{real-data} displays not only the color-coded normalized graphon mutual information estimates between lunch hours but also includes several rows demonstrating the consistency of the graphon mutual information behavior throughout the days. This effectively shows that the graphon mutual information estimator accurately captures key properties of the graph such as connectedness, symmetry, and dynamics, confirming the findings of~\cite{mastrandrea2015contact}.
\begin{figure}[!htb]
    \centering
\includegraphics[width=0.9
\textwidth]{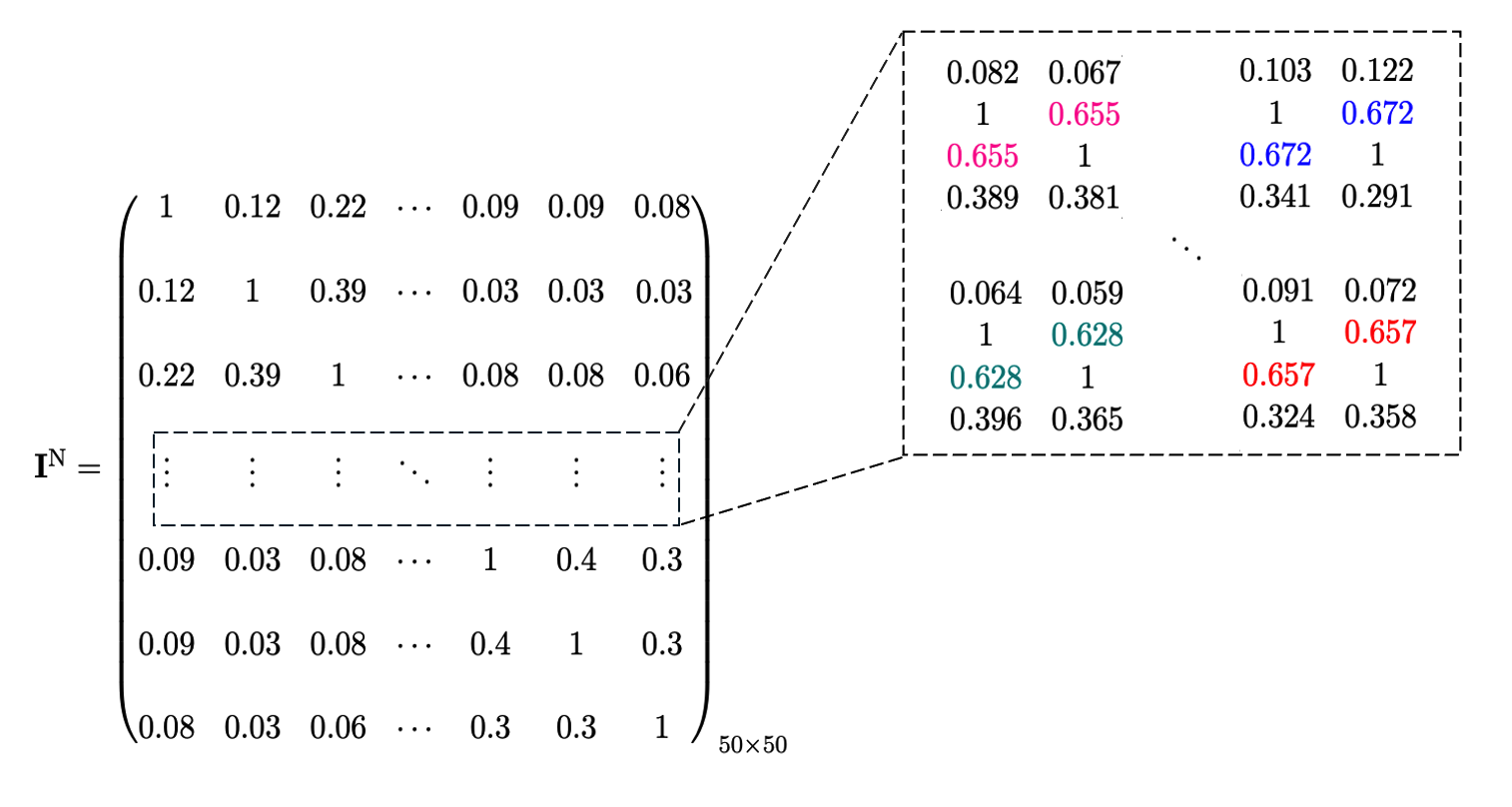}
    \caption{Snapshot from the normalized estimated graphon mutual information matrix between hourly student interactions depicting lunch hours with colored numbers.}
    \label{real-data}
\end{figure}
\subsection{Graphon total correlation and graphon conditional total correlation}
As discussed in Section~\ref{sec-info-measure-three}, there is often interest in quantifying higher--order interdependencies among more than two random variables. In this context, we estimate the shared information among three observed adjacency matrices that share the same node set. We also aim to determine which graph plays the most significant role in the trivariate setting by assessing various conditional dependency scenarios. This analysis enables us to understand how each graph affects the relationships among the other two.

To better understand the interdependence structure of the multiplex temporal graph described in~\cite{mastrandrea2015contact}, we begin by examining the normalized graphon mutual information matrix. We then select the three consecutive hours with the highest graphon mutual information, as we expect these to exhibit high shared information.  Furthermore, we include some randomly chosen hours to illustrate the differences in the outputs of the estimated graphon total correlation, graphon interaction information, and conditional graphon total correlation between meaningful scenarios and randomly selected times. Adjacency matrices $A^{(11)}, A^{(12)}, A^{(13)}$ and $A^{(21)}, A^{(22)}, A^{(23)}$, denote the adjacency matrices of three consecutive hours of the second and the third day, respectively. We estimate the graphon total correlation among $A^{(11)}, A^{(12)}, A^{(13)}$ and $A^{(21)}, A^{(22)}, A^{(12)}$, and show how we obtain different outputs for the conditional graphon total correlation estimator when we vary the graph on which we condition. The estimated graphon total correlation and conditional graphon total correlation at a selection of time points is given by:

\begin{align}
\nonumber & \widehat{\mathrm{TC}}^{\mathrm{N}}_{\mathbf{W}_{11,12,13}}(W^{(11)},W^{(12)},W^{(13)})=0.565 ,\\
\label{tc-111213}
&\widehat{\mathrm{TC}}^{\mathrm{N}}_{\mathbf{W}_{11,12,13}}(W^{(11)},W^{(13)}|W^{(12)})=0.204, \quad \widehat{\mathrm{TC}}^{\mathrm{N}}_{\mathbf{W}_{11,12,13}}(W^{(11)},W^{(12)}|W^{(13)})=0.620.\\ 
    &\nonumber \widehat{\mathrm{TC}}^{\mathrm{N}}_{\mathbf{W}_{21,22,23}}(W^{(21)},W^{(22)},W^{(23)})=0.555,\\
    &\label{tc-212223}\widehat{\mathrm{TC}}^{\mathrm{N}}_{\mathbf{W}_{21,22,23}}(W^{(21)},W^{(23)}|W^{(22)})=0.209, \quad \widehat{\mathrm{TC}}^{\mathrm{N}}_{\mathbf{W}_{21,22,23}}(W^{(21)},W^{(22)}|W^{(23)})=0.512.
\end{align}
   Equation~\eqref{tc-111213} therefore illustrates three dependence relations. The estimated normalized total correlation among three consecutive hours in a day is relatively high, at $0.565$, as is the interaction information, $\widehat{\mathrm{I}}_{\mathbf{W}_{11,12,13}}(W^{(11)};W^{(12)};W^{(13)})=0.11$, lying in the range $[-0.013,0.119]$ following Corollary~\ref{cor:interaction-information-bounds}. The estimated graphon interaction information is very close to its upper bound, thus suggesting redundancy. We observe that when conditioning on the generating mechanism of $A^{(12)}$, the relationship between the generating mechanisms of $A^{(11)}$ and $A^{(13)}$ is closer to zero, unlike when conditioned on the generating mechanism of $A^{(13)}$. This suggests that $A^{(12)}$ largely accounts for the shared information between $A^{(11)}$ and $A^{(13)}$, a conclusion that is not evident when conditioning on $A^{(13)}$. Additionally, the normalized mutual information between the generating mechanisms of $A^{(11)}$ and $A^{(12)}$ is given as $\widehat{\mathrm{I}}^{\mathrm{N}}_{\mathbf{W}{11,12}}(W^{(11)};W^{(12)})=0.655$, and the conditional total correlation when conditioned on the generating mechanism of $A^{(13)}$ is given as $\widehat{\mathrm{TC}}^{\mathrm{N}}_{\mathbf{W}_{11,12,13}}(W^{(11)},W^{(12)}|W^{(13)})=0.62$. From this, we observe that $A^{(13)}$ has almost no impact in explaining the shared information between $A^{(11)}$ and $A^{(12)}$ since the estimated conditional graphon total correlation is very close to the estimated graphon mutual information.
\vspace{-0.5pt}
We observe a similar phenomenon in Equation~\eqref{tc-212223}, where the graphon total correlation and interaction information among the adjacency matrices of three consecutive hours are again comparably high, with the estimated graphon interaction information equaling $0.109$, closely approaching the upper bound of $0.12$. Additionally, conditioning on the intermediate adjacency matrix substantially explains the shared information between the first and third adjacency matrices. Similarly, when conditioning on the third adjacency matrix, we find that the estimated conditional graphon total correlation, given as $0.512$, is again very close to the estimated graphon mutual information between the generating mechanisms of $A^{(11)}$ and $A^{(12)}$, given as $\widehat{\mathrm{I}}^{\mathrm{N}}_{\mathbf{W}_{11,12}}(W^{(11)};W^{(12)}) = 0.598$. These findings align well with our understanding of time, suggesting that $A^{(t)}$ typically explains the relationship between $A^{(t-1)}$ and $A^{(t+1)}$ for $t>0$. We also investigate a scenario in which we randomly select some adjacency matrices to estimate the graphon total correlation and the conditional graphon total correlation for various dependence combinations, as follows:

   \begin{align}
    \nonumber & \widehat{\mathrm{TC}}^{\mathrm{N}}_{\mathbf{W}_{11,17,28}}(W^{(11)},W^{(17)},W^{(28)})=0.236 \\
&\label{tc-1117}\widehat{\mathrm{TC}}^{\mathrm{N}}_{\mathbf{W}_{11,17,28}}(W^{(11)},W^{(17)}|W^{(28)})=0.292, \quad \widehat{\mathrm{I}}^{\mathrm{N}}_{\mathbf{W}_{11,17}}(W^{(11)};W^{(17)})=0.295 \\
 \label{tc-1128} &\widehat{\mathrm{TC}}^{\mathrm{N}}_{\mathbf{W}_{11,17,28}}(W^{(11)},W^{(28)}|W^{(17)})=0.109, \quad \widehat{\mathrm{I}}^{\mathrm{N}}_{\mathbf{W}_{11,28}}(W^{(11)};W^{(28)})=0.118\\
\label{tc-1728} &\widehat{\mathrm{TC}}^{\mathrm{N}}_{\mathbf{W}_{11,17,28}}(W^{(17)},W^{(28)}|W^{(11)})=0.075, \quad \widehat{\mathrm{I}}^{\mathrm{N}}_{\mathbf{W}_{17,28}}(W^{(17)};W^{(28)})=0.078.
   \end{align}
Equations~\eqref{tc-1117}--\eqref{tc-1728} illustrate various conditional dependency scenarios, demonstrating that when graphs are randomly selected or lack explanatory roles, conditioning on one does not provide more information than the mutual information between the generating mechanisms between the other two graphs.

\section{Conclusion}\label{sec-conclusion}
In conclusion, this paper developed a comprehensive theory of multivariate graphon information--theoretic measures for multipex exchangeable random graphs, or graphs that are subject to graph matching. We introduced joint graphon entropy, and the graphon mutual information between two graphs, while also proposing consistent estimators via consistent estimation of the bivariate graph limit model. We also introduced the graphon mutual information matrix for multiplex graphs, and defined the von Neumann entropy of the graphon mutual information density matrix, which is a normalized version of the graphon mutual information matrix. We formulated the trivariate graph limit model, and introduced the joint graphon entropy among three exchangeable random graphs. Consequently, we introduced the graphon interaction information among three graphs, and discussed the concepts of synergy and redundancy for graphs for the first time. Furthermore, we introduced multivariate graphon information--theoretic measures based on relative entropy for assessing independence and conditional independence among the $d$ graph-generating mechanisms: graphon total correlation, graphon dual total correlation, their conditional versions, and graphon $O$-information. We introduced consistent estimators for these measures in the setting of three graphs and discussed extensions of estimation for $d \geq 3$. We assessed the performance of our estimators in synthetic graphs by demonstrating a decrease in root mean squared error for all proposed estimators as the sample size increased. Additionally, we explored the relationship between multiplex graph structure, synergy, redundancy and the introduced multivariate graphon information--theoretic measures. In real-world networks, we observed the utility of graphon mutual information, graphon interaction information, graphon total correlation, and conditional graphon total correlation in characterising the complexity of time-varying graphs and in depicting time intervals that share the most information. This aligns with previous non-mathematical studies about the dataset~\cite{mastrandrea2015contact}, underscoring our theory's potential in addressing complex multiplex network inference problems.
\section*{Acknowledgements}
The authors thank Charles Dufour for his help with the implementation of the software package.

\newpage
\bibliographystyle{acm}

\newpage 
\section*{Appendices}
\section*{I \quad Proofs}
\subsection*{I.A\quad Proof of Corollary 1}
We note that the univariate case, i.e. when the object is a single graph, the analogous of this theorem has been proved by~\cite{janson2013graphons}. We extend the proof for the setting of the joint entropy between two exchangeable graphs $G_1$ and $G_2$ generated by graphons $W^{(1)}$ and $W^{(2)}$, respectively.
\begin{proof}
    Let $\mathcal{H}(\cdot)$ denote the Shannon entropy. We note that $A^{(1)}$ and $A^{(2)}$ share the same latent vector. Therefore, conditioning on $\xi_1,\ldots,\xi_n$, and defining $\underline{A}_{ij}=[A_{ij}^{(1)}\quad A_{ij}^{(2)}]$, we get $\underline{A}_{ij}|\xi \sim \mathrm{MultBern_2}(\underline{W})$, where $\underline{W}=[W^{(1)} \quad W^{(2)} \quad W^{(12)}]$. \\ Let $\mathbf{W}_{1,2}=[W^{(12)} \quad (W^{(1)}-W^{(12)}) \quad (W^{(2)}-W^{(12)}) \quad (1-W^{(1)}-W^{(2)}+W^{(12)})]$ specify the parameters of the bivariate graph limit model. It follows that
    \begin{align}
        \nonumber &\mathcal{H}(G_1(n,W^{(1)}),G_2(n,W^{(2)})|\xi_i,\dots,\xi_n)=\sum_{i<j}\mathcal{H}(\underline{A}_{ij}|\xi_i,\ldots,\xi_n)=\sum_{i<j}\mathcal{H}(\mathrm{MultBern_2(\underline{W})})\\
       \nonumber &=\sum_{i<j}\{(W^{(12)}(\xi_i,\xi_j))\log(W^{(12)}(\xi_i,\xi_j))+(W^{(1)}(\xi_i,\xi_j)-W^{(12)}(\xi_i,\xi_j))\log(W^{(1)}(\xi_i,\xi_j)-W^{(12)}(\xi_i,\xi_j))\\
       \nonumber &+(W^{(2)}(\xi_i,\xi_j)-W^{(12)}(\xi_i,\xi_j))\log(W^{(2)}(\xi_i,\xi_j)-W^{(12)}(\xi_i,\xi_j))\\
       &+(1-W^{(1)}(\xi_i,\xi_j)-W^{(2)}(\xi_i,\xi_j)+W^{(12)}(\xi_i,\xi_j))\log(1-W^{(1)}(\xi_i,\xi_j)-W^{(2)}(\xi_i,\xi_j)+W^{(12)}(\xi_i,\xi_j))\}.
    \end{align}
Hence, 
\begin{align}
    \nonumber &\mathcal{H}(G_1(n,W^{(1)}),G_2(n,W^{(2)})\geq \mathbb{E}[\mathcal{H}(G_1(n,W^{(1)}),G_2(n,W^{(2)})|\xi] \\
   \nonumber &=\mathbb{E}\sum_{i<j}[(W^{(12)}(\xi_i,\xi_j))\log(W^{(12)}(\xi_i,\xi_j))+(W^{(1)}(\xi_i,\xi_j)-W^{(12)}(\xi_i,\xi_j))\log(W^{(1)}(\xi_i,\xi_j)-W^{(12)}(\xi_i,\xi_j))\\
       \nonumber &+(W^{(2)}(\xi_i,\xi_j)-W^{(12)}(\xi_i,\xi_j))\log(W^{(2)}(\xi_i,\xi_j)-W^{(12)}(\xi_i,\xi_j))\\
       &+(1-W^{(1)}(\xi_i,\xi_j)-W^{(2)}(\xi_i,\xi_j)+W^{(12)}(\xi_i,\xi_j))\log(1-W^{(1)}(\xi_i,\xi_j)-W^{(2)}(\xi_i,\xi_j)+W^{(12)}(\xi_i,\xi_j))]\\
        &={n \choose 2}\mathcal{H}(\mathbf{W}_{1,2}).
    \end{align}
    We now need to obtain an upper bound, and proceed similarly to the proof of the setting involving only one graph. We fix an integer $m$ and let $M_i=\lceil m\xi_i\rceil$. Therefore, $M_i=k \iff X_i \in I_{km}$. We then have
    \begin{align}
        \nonumber \mathcal{H}(G_1(n,W^{(1)}),G_2(n,W^{(2)})) &\leq \mathcal{H}(G_1(n,W^{(1)}),G_2(n,W^{(2)}),M_1,\ldots,M_n)\\
        &=\mathcal{H}(M_1,\ldots,M_n)+\mathcal{H}(G_1(n,W^{(1)}),G_2(n,W^{(2)})|M_1,\ldots,M_n)
    \end{align}
    Given that $M_1,\ldots,M_n\overset{\mathrm{iid}}{\sim} U\{1,\ldots,m\}$, consequently
    \begin{equation}
        \mathcal{H}(M_1,\ldots,M_n)=\sum_{i=1}^n \mathcal{H}(M_i)=n\log m.
    \end{equation}
    Furthermore,
    \begin{equation}
        \mathcal{H}(G_1(n,W^{(1)}),G_2(n,W^{(2)})|M_1,\ldots,M_n)\leq \sum_{i<j}\mathcal{H}(\underline{A}_{ij}|M_1,\ldots,M_n)=\sum_{i<j}\mathcal{H}(\underline{A}_{ij}|M_i,M_j).
    \end{equation}
    Next, we define for $k,l=1,\ldots,m$, and $(i)=(1),(2),(12)$
    \begin{align*}
        w_m^{(i)}(k,l)=\mathbb{E}[W^{(i)}(\xi_1,\xi_2)|M_1=k,M_2=l]=m^2\int_{I_{km}}\int_{I_{lm}}W^{(i)}(x,y)\ dx dy,
    \end{align*}
    the average of $W$ over $I_{km}\times I_{lm}$, and let
    \begin{equation*}
        W_m^{(i)}(x,y):=w_m^{(i)}(k,l) \quad \mathrm{if} \quad x\in I_{km} \quad \mathrm{and} \quad y\in I_{lm}.
    \end{equation*}
Thus, $W_m(\xi_i,\xi_j)$ equals $\mathbb{E}[W(\xi_1,\xi_2)|M_1,M_2]$. Given $M_i=k$ and $M_j=l$,
\begin{align*}
    &\mathbb{P}(A^{(1)}_{ij}=1)=\mathbb{E}[W^{(1)}(\xi_1,\xi_2)|M_1=k,M_2=l]=w_m^{(1)}(k,l),\\
    &\mathbb{P}(A^{(2)}_{ij}=1)=\mathbb{E}[W^{(1)}(\xi_1,\xi_2)|M_1=k,M_2=l]=w_m^{(2)}(k,l),\\
    &\mathbb{P}(A^{(1)}_{ij}=1,A_{ij}^{(2)}=1)=\mathbb{E}[W^{(12)}(\xi_1,\xi_2)|M_1=k,M_2=l]=w_m^{(12)}(k,l),\\
    &\mathbb{P}(A^{(1)}_{ij}=1,A_{ij}^{(2)}=0)=\mathbb{E}[W^{(1)}(\xi_1,\xi_2)-W^{(12)}(\xi_i,\xi_j)|M_1=k,M_2=l]=w_m^{(1)}(k,l)-w_m^{(12)}(k,l),\\
    &\mathbb{P}(A^{(1)}_{ij}=0,A_{ij}^{(2)}=1)=\mathbb{E}[W^{(2)}(\xi_1,\xi_2)-W^{(12)}(\xi_i,\xi_j)|M_1=k,M_2=l]=w_m^{(2)}(k,l)-w_m^{(12)}(k,l),\\
    &\mathbb{P}(A^{(1)}_{ij}=0,A_{ij}^{(2)}=0)=\mathbb{E}[1-W^{(1)}(\xi_1,\xi_2)-W^{(2)}(\xi_i,\xi_j)+W^{(12)}(\xi_i,\xi_j)|M_1=k,M_2=l]\\
    &=1-w_m^{(1)}(k,l)-w_m^{(2)}(k,l)+w_m^{(12)}(k,l).
\end{align*}
Therefore, \begin{align*}
    &\mathcal{H}(\underline{A_{ij}}|M_i=k,M_j=k)\\
    &=(w_m^{(12)}(k,l))\log(w_m^{(12)}(k,l))+(w_m{(1)}(k,l)-w_m^{(12)}(k,l))\log(w_m{(1)}(k,l)-w_m^{(12)}(k,l))\\
    &+(w_m^{(2)}(k,l)-w_m^{(12)}(k,l))\log(w_m^{(2)}(k,l)-w_m^{(12)}(k,l))\\
    &+(1-w_m^{(1)}(k,l)-w_m^{(2)}(k,l)+w^{(12)}_m(k,l))\log(1-w_m^{(1)}(k,l)-w_m^{(2)}(k,l)+w^{(12)}_m(k,l)).
\end{align*}
Consequently,
\begin{align}
    \nonumber &\mathbb{E}[\mathcal{H}(\underline{A_{ij}}|M_i,M_j)]\\
    \nonumber &=m^{-2}\sum_{k=1}^m (w_m^{(12)}(k,l))\log(w_m^{(12)}(k,l))+(w_m{(1)}(k,l)-w_m^{(12)}(k,l))\log(w_m{(1)}(k,l)-w_m^{(12)}(k,l))\\
    \nonumber &+(w_m^{(2)}(k,l)-w_m^{(12)}(k,l))\log(w_m^{(2)}(k,l)-w_m^{(12)}(k,l))\\
    \nonumber &+(1-w_m^{(1)}(k,l)-w_m^{(2)}(k,l)+w^{(12)}_m(k,l))\log(1-w_m^{(1)}(k,l)-w_m^{(2)}(k,l)+w^{(12)}_m(k,l))\\
    \nonumber &=\iint_{[0,1]^2}\Big\{(W_m^{(12)}(x,y))\log(W_m^{(12)}(x,y))\\
    \nonumber &+(W_m^{(1)}(x,y)-W_m^{(12)}(x,y))\log(W_m{(1)}(x,y)-W_m^{(12)}(x,y))\\
    \nonumber &+(W_m^{(2)}(x,y)-W_m^{(12)}(x,y))\log(W_m^{(2)}(x,y)-W_m^{(12)}(x,y))\\
     &+(1-W_m^{(1)}(x,y)-W_m^{(2)}(x,y)+W^{(12)}_m(x,y))\log(1-W_m^{(1)}(x,y)-W_m^{(2)}(x,y)+W^{(12)}_m(x,y))\Big\} dx dy.
\end{align}
Combining Equations (24)-(27), we obtain
\begin{align*}
    &\mathcal{H}(G_1(n,W^{(1)}),G_2(n,W^{(2)}))\\
    &\leq n\log m+ {n \choose 2}\iint_{[0,1]^2}\Big\{\{(W_m^{(12)}(x,y))\log(W_m^{(12)}(x,y))\\
    \nonumber &+(W_m^{(1)}(x,y)-W_m^{(12)}(x,y))\log(W_m{(1)}(x,y)-W_m^{(12)}(x,y))\\
    \nonumber &+(W_m^{(2)}(x,y)-W_m^{(12)}(x,y))\log(W_m^{(2)}(x,y)-W_m^{(12)}(x,y))\\
     &+(1-W_m^{(1)}(x,y)-W_m^{(2)}(x,y)+W^{(12)}_m(x,y))\log(1-W_m^{(1)}(x,y)-W_m^{(2)}(x,y)+W^{(12)}_m(x,y))\Big\} \ dx dy,
\end{align*}
and therefore, $\forall m \geq 1$,
\begin{align*}
    &\lim_{n \to \infty} \sup {n \choose 2}^{-1} \mathcal{H}(G_1(n,W^{(1)}),G_2(n,W^{(2)})) \\
    &\leq \iint_{[0,1]^2}\Big\{\{(W_m^{(12)}(x,y))\log(W_m^{(12)}(x,y))\\
    \nonumber &+(W_m^{(1)}(x,y)-W_m^{(12)}(x,y))\log(W_m^{(1)}(x,y)-W_m^{(12)}(x,y))\\
    \nonumber &+(W_m^{(2)}(x,y)-W_m^{(12)}(x,y))\log(W_m^{(2)}(x,y)-W_m^{(12)}(x,y))\\
     &+(1-W_m^{(1)}(x,y)-W_m^{(2)}(x,y)+W^{(12)}_m(x,y))\log(1-W_m^{(1)}(x,y)-W_m^{(2)}(x,y)+W^{(12)}_m(x,y))\Big\} \ dx dy.
\end{align*}
Finally, let $m\to \infty$. Then, $W_{m}^{(i)}(x,y)\to W^{(i)}(x,y)$ a.e., and thus $\mathrm{RHS} \to \mathcal{H}({\mathbf{W}_{1,2}})$ by dominated convergence.
\end{proof}

\subsection*{I. B $\quad$ Proof of Theorem 2}
\begin{proof}
    We bound the magnitude difference between the estimated graphon and the true graphon, for each graphon separately. For ease of notation let $W^{(d)}=W^{(d)}(x,y)$ for $(d)=(1),(2),(12)$.
    We demonstrate the general idea for one term and the rest follows:
    \begin{align*}
        &\Big|\iint _{[0,1]^2}\widehat{W}^{(12)}\log\widehat{W}^{(12)}-W^{(12)}\log W^{(12)} \ dx \ dy \Big|\\
        & = \Big|\iint _{[0,1]^2}(\widehat{W}^{(12)}-W^{(12)}) \log \widehat{W}^{(12)} + W^{(12)} \log\Big(\frac{\widehat{W}^{(12)}}{W^{(12)}}\Big) \ dx \ dy \Big|\\
         & \leq \Big|\iint _{[0,1]^2}(\widehat{W}^{(12)}-W^{(12)}) \log \widehat{W}^{(12)} \ dx \ dy\Big| + \Big|\iint_{[0,1]^2}W^{(12)}\frac{\widehat{W}^{(12)}-W^{(12)}}{W^{(12)}} \ dx \ dy \Big| \\
         &\leq \Big|\iint _{[0,1]^2}\Big(\widehat{W}^{(12)}-W^{(12)}\Big) \Big(\log W^{(12)}_{0}+\frac{(\widehat{W}^{(12)}-W^{(12)})}{W^{(12)}_0}+\frac{1}{2}(\widehat{W}^{(12)}-W^{(12)})^2\Big)\ dx \ dy \Big|\\
         &+\Big|\iint _{[0,1]^2}(\widehat{W}^{(12)}-W^{(12)}) \ dx \ dy \Big| \\
         &\leq C_1 \sqrt{\iint _{[0,1]^2}(\widehat{W}^{(12)}-W^{(12)})^2 \ dx \ dy}+ C_2 \iint _{[0,1]^2}(\widehat{W}^{(12)}-W^{(12)})^2 \ dx \ dy \\
         &\leq C\Big(n^{-2\alpha/\alpha+1}+\frac{\log n}{n}+\frac{1}{n^{\alpha}}\Big)^{1/2}
\end{align*}

    with probability at least $1-\exp(-C'n)$, where the result follows from the rate-optimal graphon estimation~\cite{gao2015rate}. The same technique applies to the remaining terms.
\end{proof}
\section*{II \quad Background and Auxillary Results}\label{sec-graphon-est}
\subsection*{Graphon Estimation}
 A commonly adopted approach in graphon estimation is approximating the graphon $W(x,y)$, $0<x,y<1$ by a Stochastic Block Model~\cite{Olhede_2014,wolfe2013nonparametric,gao2015rate,klopp2017oracle}. 
\begin{definition}[$k-$block Stochastic Block Model]\label{blockmodel}
Let $z$ be an $n$-length vector taking values in $\{1,\dots, k\}^n$. The Stochastic Block Model for an $n\times n$ array $A$ is specified in terms of the connectivity matrix $\mathbf{\Theta}$ as
\[\mathbb{P}(A_{ij}=1\,|\,z\}=\theta_{z_i,z_j},\quad 1\leq i<j\leq n.\]
\end{definition}
It has been shown by~\cite{Olhede_2014} that $W$ can be consistently estimated whenever it is H\"older continuous, and the optimal convergence rate has been provided by~\cite{gao2015rate,klopp2017oracle}.
An estimator $\widehat{W}$ can be seen as a Riemann sum approximation of $W$, and it is necessary to determine under what conditions these sums will converge. Lebesgue's criterion states that a bounded graphon on the interval $(0,1)^2$ is Riemann integrable if it is continuous everywhere. A way to ensure this continuity is if $W$ is $\alpha-$Hölder continuous for some $0<\alpha \leq 1$, indicating that
\begin{equation}\label{holder-graphon}
    W\in\mathrm{H\ddot{o}lder}^{\alpha}(M)\iff \sup_{(x,y)\neq(x',y')\in(0,1)^2} \frac{|W(x,y)-W(x',y')|}{|(x,y)-(x',y')|^{\alpha}}\leq M<\infty
\end{equation}
 H\"older contuinity, a necessary condition to ensure tractable estimation, is a standard assumption in the literature of nonparametric graphon estimation, see~\cite{Olhede_2014,gao2015rate,klopp2017oracle,wolfe2013nonparametric} for details and convergence rates of nonparametric graphon estimation. We now detail graphon estimation via stochastic block model approximation as treated by~\cite{Olhede_2014}. 

 Given a single adjacency matrix $\mathbf{A}\in \{0,1\}^{n \times n}$, we estimate $W(x,y) \in \mathcal{F}^{\alpha(M)}$ up to the rearrangement of the axes, using a stochastic block model with a single community size $h$, which can be either determined by the user or automatically determined from the adjacency matrix. Given this notation, then we have $n=hk+r$, for integers $h,k,r$, where $k$, is the number of blocks chosen as $k=\lceil n^{1/[(\alpha \wedge 1)+1]}\rceil$ following~\cite{gao2015rate,Olhede_2014}, $h\in\{2,\ldots,n\}$ is the bandwidth and $r=n \ \mathrm{mod} \ h$ is a reminder term between $0$ and $h-1$. We introduce a community membership vector $z$ which groups together nodes that should lie in the same groups. Its components will be in $\{1,\ldots,k\}$. 
 Let $\mathcal{Z}_k \subseteq \{1,\ldots,k\}^n$ represent the set including all community assignments $z$ respecting the form of $n=hk+r$. The main difficulty lies in estimating $z$, and several methods are available for estimating $z$. Common methods include spectral clustering~\cite{von2007tutorial}, likelihood methods~\cite{celisse2012consistency, bickel2009nonparametric, amini2013pseudo,choi2012stochastic, zhao2012consistency}, and modularity maximization~\cite{newman2004finding}.
 We estimate $z$ by the method of maximum likelihood as follows,
 \begin{equation*}\label{z}
     \hat{z}=\mathrm{argmax}_{z\in \mathcal{Z}_k}\sum_{i<j}A_{ij}\log \Bar{A}_{z_i z_j}+(1-A_{ij})\log(1-\Bar{A}_{z_i z_j})\},
 \end{equation*}
 where for all $1\leq a,b\leq k$ the histogram bins are defined as follows
 \begin{equation*}\label{A_bar}
\Bar{A}_{ab}=\frac{\sum_{i<j}A_{ij}\mathbbm{1}(z_i=a)\mathbbm{1}(z_j=b)}{\sum_{i<j}\mathbbm{1}(z_i=a)\mathbbm{1}(z_j=b)}.
 \end{equation*}
 Each bin height resembles the probability of edges present in the histogram bin which corresponds to a block of Bernoulli trials given by the objective function in~\ref{z}.
 It is also evident that $\Bar{A}_{ab}=\Bar{A}_{ba}$.
 The network histogram~\cite{Olhede_2014} is given as
 \begin{equation}\label{est-graphon}
     \widehat{W}(x,y;h)=\Bar{A}_{\min\{\lceil nx/h\rceil,k\}\min\{\lceil ny/h\rceil,k\}}, \quad 0<x,y<1.
 \end{equation}

\subsection*{Estimation of the Multivariate Graphon}
We will estimate $W^{(1)}(x,y)$, $W^{(2)}(x,y)$ and $W^{(12)}(x,y)$, up to re-arrangement of the axes, given the adjacency matrices $A^{(1)}$ and $A^{(2)}$, both of size $n\times n$ using a correlated two-layer stochastic block model with a single community vector and size $h$. The community vector $z$ of length $n$ groups nodes of the graphs that should lie on the same group, and we let $\mathcal{Z}_k$ denote all possible community vectors that respect $n=hk+r$, where $h$ denotes the community size or bandwidth, $k$ the number of communities, and $r$ is a remainder term between $0$ and $h-1$. All components of $z$ take values in $[k]$. We first formulate the bivariate Bernoulli likelihood for graphon estimation via two-layer correlated stochastic block model approximation for a given $z$, and then estimate $z$ via profile maximum likelihood. 

\subsubsection*{Multivariate Bernoulli Likelihood}\label{bivariate_MLE}
We now present the bivariate Bernoulli likelihood, which can be extended to trivariate and multivariate cases using the same approach. This involves constructing joint adjacency matrices, as demonstrated here for $A^{(12)}$. Using a correlated two-layer stochastic block-model~\cite{pamfil2020inference}, similarly to the case of univariate graphon estimation as detailed above, the task now simplifies to the estimation of the probability matrices $\Theta^{(1)}$, $\Theta^{(2)}$ and $\Theta^{(12)}$. To achieve this, we develop a bivariate maximum likelihood formulation, taking into account two adjacency matrices $A^{(1)}$ and $A^{(2)}$, corresponding to the graphs ${G}_1$ and ${G}_2$, respectively. It is worth restating that ${G}_1$ and ${G}_2$ are assumed to be node-aligned graphs, both of which are associated with the same underlying latent vector. This alignment is integral to the estimation process, as it allows for a unified approach in analyzing the interrelations and dependencies between the two graphs.
\begin{align*}
    &\mathbb{P}(\mathbf{A}^{(1)},\mathbf{A}^{(2)}|\mathbf{z},\theta^{(1)},\theta^{(2)},\theta^{(12)})\\
    &=\prod_{i,j}\Big\{{\theta^{(12)}_{z_i z_j}}^{A_{ij}^{(1)}A_{ij}^{(2)}}({\theta^{(1)}_{z_i z_j}}-{\theta^{(12)}_{z_i z_j}})^{A_{ij}^{(1)}(1-A_{ij}^{(2)})}({\theta^{(2)}_{z_i z_j}}-{\theta^{(12)}_{z_i z_j}})^{(1-A_{ij}^{(1)})A_{ij}^{(2)}} (1-{\theta^{(1)}_{z_i z_j}}-{\theta^{(2)}_{z_i z_j}}+{\theta^{(12)}_{z_i z_j}})^{(1-A_{ij}^{(1)})(1-A_{ij}^{(2)})}\Big\}
\end{align*}
Following~\cite{pamfil2020inference}, we define $$e_{ab}^{\alpha\beta}:=|\{(i,j)\in E: A_{ij}^1=\alpha,A_{ij}^2=\beta, z_i=a,z_j=b\}|,$$
for $(\alpha,\beta)\in\{(1,1),(1,0),(0,1),(0,0)\}$ to simplify the log--likelihood. 
Then, the log-likelihood is as follows
\begin{align*}
    \mathcal{L}&=\sum_{a,b} \{e_{ab}^{11}\log \theta^{(12)}_{ab}+e_{ab}^{10}\log(\theta^{(1)}_{ab}-\theta^{(12)}_{ab})+e_{ab}^{01}(\theta^{(2)}_{ab}-\theta^{(12)}_{ab})+e_{ab}^{00}(1-\theta^{(1)}_{ab}-\theta^{(2)}_{ab}+\theta^{(12)}_{ab})\}.
\end{align*}
Maximizing the likelihood
\begin{align*}
    \hat{\theta}^{(1)}_{ab}=\frac{e_{ab}^{11}+e_{ab}^{10}}{e_{ab}^{11}+e_{ab}^{10}+e_{ab}^{01}+e_{ab}^{00}}=\frac{m_{ab}^1}{e_{ab}},
\end{align*}

\begin{align*}
    \hat{\theta}^{(2)}_{ab}=\frac{e_{ab}^{11}+e_{ab}^{01}}{e_{ab}^{11}+e_{ab}^{10}+e_{ab}^{01}+e_{ab}^{00}}=\frac{m_{ab}^2}{e_{ab}},
\end{align*}

\begin{align*}
    \hat{\theta}^{(12)}_{ab}=\frac{e_{ab}^{11}}{e_{ab}^{11}+e_{ab}^{10}+e_{ab}^{01}+e_{ab}^{00}}=\frac{e_{ab}^{11}}{e_{ab}},
\end{align*}
where $m_{ab}^1$, $m_{ab}^2$ denote the number of edges between communities $a$ and $b$ in graphs ${G}_1$, ${G}_2$, respectively, and $e_{ab}$ denotes the number of possible edges between nodes in communities $a$ and $b$.
Going back to the notation we use, the following estimators could also be written as:
\begin{align*}
    &\hat{\theta}^{(1)}_{ab}=\frac{\sum_{i<j}A^{(1)}_{ij}\mathbbm{1}(z_i=a)\mathbbm{1}(z_j=b)}{\mathbbm{1}(z_i=a)\mathbbm{1}(z_j=b)}=\Bar{A}^{(1)}_{ab},\\
    &\hat{\theta}^{(2)}_{ab}=\frac{\sum_{i<j}A^{(2)}_{ij}\mathbbm{1}(z_i=a)\mathbbm{1}(z_j=b)}{\mathbbm{1}(z_i=a)\mathbbm{1}(z_j=b)}=\Bar{A}^{(2)}_{ab},\\
    &\hat{\theta}^{(12)}_{ab}=\frac{\sum_{i<j}A^{(12)}_{ij}\mathbbm{1}(z_i=a)\mathbbm{1}(z_j=b)}{\mathbbm{1}(z_i=a)\mathbbm{1}(z_j=b)}=\Bar{A}^{(12)}_{ab}.
\end{align*}

The main difficulty is in estimating the community vector $z$. Following the approach adopted in~\cite{Olhede_2014} for the case of a single graph, we estimate the shared community membership vector $z$ between adjacency matrices $A^{(1)}$ and $A^{(2)}$ by the method of profile maximum likelihood as follows:
\begin{align}
   \nonumber \hat{z}:&=\argmax_{z \in \mathcal{Z}_k}\sum_{i<j}\{A_{ij}^{(1)} A_{ij}^{(2)}\log (\Bar{A}_{z_i z_j}^{(12)})+A_{ij}^{(1)}(1-A_{ij}^{(2)})\log(\Bar{A}_{z_i z_j}^{(1)}-\Bar{A}_{z_i z_j}^{(12)})\\
  &+(1-A_{ij}^{(1)})A_{ij}^{(2)}\log(\Bar{A}_{z_i z_j}^{(2)}-\Bar{A}_{z_i z_j}^{(12)})+(1-A_{ij}^{(1)})(1-A_{ij}^{(2)})\log(1-\Bar{A}_{z_i z_j}^{(1)}-\Bar{A}_{z_i z_j}^{(2)}+\Bar{A}_{z_i z_j}^{(12)})\},
\end{align}
where for all $1\leq a,b \leq k$, and $k=\lceil n^{1/[(\alpha \wedge 1)+1]}\rceil$,
\begin{align}
&\Bar{A}_{ab}^{(d)}=\frac{A^{(d)}_{ij}\mathbbm{1}(\hat{z}_i=a)\mathbbm{1}(\hat{z}_j=b)}{\mathbbm{1}(\hat{z}_i=a)\mathbbm{1}(\hat{z}_j=b)},
\end{align}
for $(d)=(1),(2),(12)$, and $A^{(12)}$ as detailed in the paper equals $A^{(1)}_{ij}A^{(2)}_{ij}$. Since all adjacency matrices are symmetric, we have $\Bar{A}^{(d)}_{ab}=\Bar{A}^{(d)}_{ba}$.
Combining \eqref{z-est}-\eqref{theta-est}, and following~\cite{Olhede_2014}, we define the estimators of $W^{{(1)}}$, $W^{{(2)}}$ and $W^{{(12)}}$ as follows
\begin{align*}
    &\widehat{W}^{(1)}(x,y;h):=\Bar{A}^{(1)}_{\min(\lceil nx/h\rceil,k) \min(\lceil ny/h\rceil,k)}, \ 0<x,y<1,\\
    &\widehat{W}^{(2)}(x,y;h):=\Bar{A}^{(2)}_{\min(\lceil nx/h\rceil,k) \min(\lceil ny/h\rceil,k)}, \ 0<x,y<1,\\
    &\widehat{W}^{(12)}(x,y;h):=\Bar{A}^{(12)}_{\min(\lceil nx/h\rceil,k) \min(\lceil ny/h\rceil,k)}, \ 0<x,y<1.
\end{align*}
This estimation method can analogously be extended to the setting of three and more graphs.

\section*{III \quad Additional Simulation Results}
\subsection*{Graphon Mutual Information Matrix}
We now investigate the von Neumann entropy calculated from the spectrum of the graphon mutual information density matrix for two different scenarios.
\subsubsection*{Case 1: Percolated graphs}
Consider adjacency matrices $A^{(1)}$, $A^{(2)}$ and $A^{(3)}$ of graphs ${G}_1(n,W^{(1)})$, ${G}_2(n,W^{(2)})$ and ${G}_3(n,W^{(3)})$ generated from graph limits $W^{(1)}(x,y)=xy$, $W^{(2)}(x,y)=0.95*xy$ and $W^{(3)}(x,y)=0.9*xy$, respectively. The graphon mutual information matrix $\mathbf{I}$, without normalization, is then given as on the left below, and the normalization following Equations 6 and 7 given in the paper, that gives rise to the graphon mutual information density matrix $\rho$, is as given on the right:
$$
\mathbf{I}=
\begin{pmatrix}
0.4275 & 0.3993 & 0.3409 \\
0.3993 & 0.4260 & 0.3316 \\
0.3409 & 0.3316 &  0.4230  
\end{pmatrix}, \quad
\mathbf{I}^{\rho}=
\begin{pmatrix}
1/3 & 0.2948 & 0.2677 \\
0.2948 & 1/3 & 0.2678 \\
0.267 & 0.2678 & 1/3 
\end{pmatrix}.
$$
The eigenvalues associated with the graphon mutual information density matrix $\rho$ are $\lambda_1\approx 0.8871, \ \lambda_2\approx 0.0744$ and $\lambda_3 \approx 0.0385$. Redundancy in this system can be observed from the difference in magnitude of the first eigenvalue and the rest. The von Neumann entropy associated with this matrix is calculated as $\mathcal{H}(\mathbf{I}^{\rho})=-\sum_{i=1}^3 \lambda_i \log \lambda_i \approx 0.4152$. A final normalization for the von Neumann entropy to ensure it is normalized by the maximal value this entropy can attain gives $\mathcal{H}^{\mathrm{N}}(\mathbf{I}^{\rho})\approx 0.5474/\log(3)\approx 0.377$, which is closer to the so--called pure--state configuration as discussed in the paper. 

\subsubsection*{Case 2: Three different graph limits}
In this case, we have the following graph limits $W^{(1)}(x,y)=xy$, $W^{(2)}(x,y)=0.3\exp(-0.5(x+y))$ and $W^{(3)}(x,y)=0.3(x
+y)$, which are different in nature, and do not seem to share a lot of information apart from the same latents. We calculate the graphon mutual information between the adjacency matrices generated by these graphs and consequently obtain the following graphon mutual information matrix, and graphon mutual information density matrix:
\begin{equation*}
\mathbf{I}=
\begin{pmatrix}
0.4275 & 0.0944 & 0.2699 \\
0.0944 & 0.2609 & 0.1066 \\
0.2699 & 0.1066 & 0.5712 
\end{pmatrix}, \quad
\mathbf{I}^{\rho}=\begin{pmatrix}
1/3 & 0.1206 & 0.2104 \\
0.1206 & 1/3 & 0.1361 \\
0.2104 & 0.1361 & 1/3 
\end{pmatrix}.
\end{equation*}
Given the different functional form of the graph limits, we do not expect to see redundancy in this case, and eigenvalues of the graphon mutual information density matrix are given as $\lambda_1 \approx 0.6484$, $\lambda_2\approx 0.2296$ and $\lambda_3 \approx 0.1220$. The von Neumann entropy associated with this matrix then is $\mathcal{H}(\mathbf{I}^{\rho})=-\sum_{i=1}^3 \lambda_i \log \lambda_i \approx 0.8754$. A final normalization for the von Neumann entropy to ensure it is normalized by the maximal value this entropy can attain gives $\mathcal{H}^{\mathrm{N}}(\mathbf{I}^{\rho})\approx 0.8754/\log(3)\approx 0.7968$, which can be seen to be close to the so-called \emph{completely mixed state}. 
\end{document}